\documentclass[11pt]{article}   
\makeatletter     
\skip\footins=4ex     
\makeatother     
\usepackage{amsmath,amsthm}   
\usepackage{amssymb}
\usepackage{lastpage}
\usepackage[pdftex]{graphicx}   
\usepackage{xcolor}
\definecolor{ao(english)}{rgb}{0.0, 0.5, 0.0}
\definecolor{raspberrypink}{rgb}{0.89, 0.31, 0.61}
\usepackage{hyperref}    
\hypersetup{colorlinks=true, citecolor=raspberrypink,linkcolor=blue}
\usepackage{epstopdf}
\usepackage{tikz}
\usepackage{enumerate}
\usepackage{mathtools}
\usepackage{upgreek}
\usepackage[font=small,labelfont=bf, margin= 0.8 in]{caption}

\newtheorem{theorem}{Theorem}
\numberwithin{theorem}{section}
\numberwithin{equation}{section}

\newtheorem{lemma}[theorem]{Lemma}
\newtheorem{corollary}[theorem]{Corollary}
\newtheorem{proposition}[theorem]{Proposition}

\newtheorem{remark}[theorem]{Remark}

\newtheorem{definition}[theorem]{Definition}

\newtheorem{innercustomthm}{Theorem}
\newenvironment{customthm}[1]
  {\renewcommand\theinnercustomthm{#1}\innercustomthm}
  {\endinnercustomthm}

\newtheorem{innercustomcor}{Corollary}
\newenvironment{customcor}[1]
  {\renewcommand\theinnercustomcor{#1}\innercustomcor}
  {\endinnercustomcor}

\newcommand{\zra}{P_a}
\newcommand{\enn}{{(n)}}

\newcommand{\hn}{\mathbb{H}^\enn}

\newcommand{\pn}{\pi^\enn}

\usepackage{maplestd2e}

\DefineParaStyle{Maple Bullet Item}
\DefineParaStyle{Maple Heading 1}
\DefineParaStyle{Maple Warning}
\DefineParaStyle{Maple Heading 4}
\DefineParaStyle{Maple Heading 2}
\DefineParaStyle{Maple Heading 3}
\DefineParaStyle{Maple Dash Item}
\DefineParaStyle{Maple Error}
\DefineParaStyle{Maple Title}
\DefineParaStyle{Maple Text Output}
\DefineParaStyle{Maple Normal}
\DefineCharStyle{Maple 2D Output}
\DefineCharStyle{Maple 2D Input}
\DefineCharStyle{Maple Maple Input}
\DefineCharStyle{Maple 2D Math}
\DefineCharStyle{Maple Hyperlink}

\title{\vspace*{-5mm} A combinatorial formula for Sahi, Stokman, and Venkateswaran's generalization of Macdonald polynomials}

\author{by {\sc Jason Saied}}

\begin{document}

\maketitle

\begin{abstract}
Sahi, Stokman, and Venkateswaran have constructed, for each positive integer $n$, a family of Laurent polynomials depending on parameters $q$ and $k$ (in addition to $\lfloor n/2\rfloor$ ``metaplectic parameters"), such that the $n=1$ case recovers the nonsymmetric Macdonald polynomials and the $q\rightarrow\infty$ limit yields metaplectic Iwahori-Whittaker functions with arbitrary Gauss sum parameters. In this paper, we study these new polynomials, which we call SSV polynomials, in the case of $GL_r$. We apply a result of Ram and Yip in order to give a combinatorial formula for the SSV polynomials in terms of alcove walks. The formula immediately shows that the SSV polynomials satisfy a triangularity property with respect to a version of the Bruhat order, which in turn gives an independent proof that the SSV polynomials are a basis for the space of Laurent polynomials. The result is also used to show that the SSV polynomials have \emph{fewer} terms than the corresponding Macdonald polynomials. We also record an alcove walk formula for the natural generalization of the permuted basement Macdonald polynomials. We then construct a symmetrized variant of the SSV polynomials: these are symmetric with respect to a conjugate of the Chinta-Gunnells Weyl group action and reduce to symmetric Macdonald polynomials when $n=1$. We obtain an alcove walk formula for the symmetrized polynomials as well. Finally, we calculate the $q\rightarrow 0$ and $q\rightarrow \infty$ limits of the SSV polynomials and observe that our combinatorial formula can be written in terms of alcove walks with only positive and negative folds respectively. In both of these $q$-limit cases, we also observe a positivity result for the coefficients. 
\end{abstract}

\section{Introduction}

In \cite{ssv}, Sahi, Stokman, and Venkateswaran constructed a new representation of the double affine Hecke algebra on a space of Laurent polynomials, generalizing the basic representation of \cite{cherednikBook} and the metaplectic Demazure-Lusztig operators of \cite{cgp}. This representation may be used to define a remarkable family of Laurent polynomials, which we call \emph{SSV polynomials}, generalizing the nonsymmetric Macdonald polynomials: this is done in \cite{ssv} for the case of $GL_r$ and in the upcoming paper \cite{ssv2} for the general case. We will use the methods of Ram and Yip \cite{ry} to give a combinatorial formula for the SSV polynomials of type $GL_r$ in terms of alcove walks. Some of our techniques are the same as those used in \cite{ssv2}. We will start by outlining the basic objects, restricting to the case of $GL_r$. 

Let $r$ be a positive integer, and let $P$ be the weight lattice corresponding to $GL_r$, which we identify with $\mathbb{Z}^r$. The generators of the (untwisted) double affine Hecke algebra $\mathbb{H}$ include elements $T_1, \dots, T_{r-1}$, $X^\mu$ for $\mu\in P$, and $Y^\mu$ for $\mu\in P$. (For a detailed discussion of double affine Hecke algebras in general, see \cite{is} or \cite{cherednikBook}. For the details of the particular case that is relevant to this paper, see Definition~\ref{daha} below.) The basic representation is a representation $\pi$ of $\mathbb{H}$ on the space of Laurent polynomials,
\[\mathbb{F}[x^{\pm 1}]=\textnormal{span}_{\mathbb{F}}\{x^\mu: \mu\in P\},\]
where $\mathbb{F}=\mathbb{C}(k,q)$ for independent parameters $k$ and $q$. We have
\begin{equation*}
    \pi(X^\lambda)x^\mu=x^{\lambda+\mu}
\end{equation*}
for all $\lambda,\mu\in P$, and the family of operators 
\begin{equation}\label{Y's}
    \{\pi(Y^\lambda): \lambda\in P\}
\end{equation}
is simultaneously diagonalizable. The basis of common eigenfunctions of (\ref{Y's}) is written as 
\begin{equation*}
    \{E_\mu: \mu\in P\}
\end{equation*}
and is called the family of \emph{nonsymmetric Macdonald polynomials}: see \cite{mac2}. The symmetric Macdonald polynomials can be derived from the nonsymmetric ones by a symmetrization procedure: see \cite{mac2} or the generalization in Section~\ref{section: symmetric} below. 

In \cite{ssv}, for a positive integer $n$, the authors considered an algebra $\mathbb{H}^\enn$ that is isomorphic to $\mathbb{H}$ (outside type $A$, $\mathbb{H}^\enn$ could be isomorphic to either $\mathbb{H}$ or the double affine Hecke algebra corresponding to the dual root system), with generators including elements $T_1, \dots, T_{r-1}$, $X^\mu$ for $\mu\in nP$, and $Y^\mu$ for $\mu\in nP$. They then constructed a representation $\pi^\enn$ of $\mathbb{H}^\enn$ on $\mathbb{F}^\enn [x^{\pm 1}]$ (where $\mathbb{F}^\enn$ is an extension of $\mathbb{F}$: see Section~\ref{section: metaplectic}) such that
\begin{equation*}
    \pi^\enn(X^\lambda)x^\mu=x^{\lambda+\mu}
\end{equation*}
for all $\lambda\in nP$ and $\mu\in P$,
and the family of operators 
\begin{equation*}
    \{\pi^\enn(Y^\lambda): \lambda\in nP\}
\end{equation*}
is simultaneously diagonalizable. The basis of common eigenfunctions 
\begin{equation*}
    \{E^\enn_\mu: \mu\in P\}
\end{equation*}
is the family of \emph{SSV polynomials}. In the case $n=1$, the SSV polynomials are the nonsymmetric Macdonald polynomials. In fact, for any positive integer $n$ and $\mu\in P$, $E^\enn_{n\mu}$ is equal to the nonsymmetric Macdonald polynomial $E_\mu$ up to a change of variables and parameters: see the final remark of \cite{ssv}. 

\noindent \textbf{Remark. }\emph{In \cite{ssv}, the authors consider a slightly more general situation, depending on a positive integer $n$, a nonzero integer $\kappa$, and $m=n/\kappa$. However, since we are only concerned with the resulting polynomials here, Proposition 5.8 of \cite{ssv} allows us to take $\kappa=1$ and $m=n$.    }

The \emph{intertwiners} are certain special elements $\tau_0^\vee, \dots, \tau_{r-1}^\vee$ in a localization of a subalgebra of $\mathbb{H}$. Their importance for the theory of Macdonald polynomials was first pointed out in the papers of Knop and Sahi \cite{knop, ks, sahiInterpolation} (for $GL_r$),  Cherednik \cite{cherednikBook} (for reduced root systems), and Sahi \cite{sahiKoornwinder} (for type BC). In particular, for all $\widehat{\nu}\in P_a$ (an affinization of $P$) and $0\leq i\leq r-1$, the intertwiners satisfy
\begin{equation}\label{intro intertwiner 1}
    \tau_i^\vee Y^{\widehat{\nu}}=Y^{s_i * \widehat{\nu}} \tau_i^\vee.
\end{equation}
After extending the representation $\pi$ to the localized algebra, this implies that if $s_i\mu\neq \mu$, then
\begin{equation}\label{intro intertwiner 2}
    \pi(\tau_i^\vee)E_\mu\sim E_{s_i\mu},
\end{equation}
where $f\sim g$ means $f$ is a nonzero multiple of $g$. (See Section~\ref{basics} for the definitions of $P_a$, $s_i$, $s_i\mu$, and $s_i * \mu$.) Since $E_0=1$, this gives a concrete way of calculating all nonsymmetric Macdonald polynomials.

In \cite{ry}, Ram and Yip give combinatorial formulas for nonsymmetric and symmetric Macdonald polynomials. These results follow from their combinatorial expansion of the product
\begin{equation*}
    \uptau_{i_1}^\vee\cdots \uptau_{i_\ell}^\vee,
\end{equation*}
where $s_{i_1}\cdots s_{i_\ell}$ is a reduced expression of an element in the affine Weyl group, in terms of alcove walks. The alcove walk model was developed by Ram \cite{ram}, following an idea of Schwer \cite{schwer}. The spirit of this idea can also be seen in a paper of Sahi \cite{sahiWeightMultiplicities}. 

In this paper, we will consider the analog of the intertwiners for $\mathbb{H}^\enn$ and use them to construct SSV polynomials by applying versions of (\ref{intro intertwiner 1}) and (\ref{intro intertwiner 2}). This technique was suggested to the author by Siddhartha Sahi and also plays a major role in \cite{ssv2}. We then apply the intertwiner product formula of \cite{ry} to obtain the following combinatorial formula for SSV polynomials, which is Theorem~\ref{main} below. 

\begin{customthm}{A}\label{intro main}
$\mu\in\mathbb{Z}^r$ uniquely determines a dominant weight $\lambda\in \mathbb{Z}^r$ and an affine Weyl group element $w$ with reduced expression $w=s_{i_1}\cdots s_{i_\ell}$ (see Theorem~\ref{main} for details) such that
\begin{equation}\label{intro E formula}
    E_\mu^\enn = \sum_{p\in \mathcal{B}(\vec{w})} a^\enn(p)x^{n\textnormal{wt}(p)+\phi(p)\lambda},
\end{equation}
where $\mathcal{B}(\vec{w})$ is the set of all alcove walks of type $(i_1,\dots,i_\ell)$ starting at the fundamental alcove 
and wt$(p)$ and $\phi(p)$ are defined in (\ref{end, weight, phi}).
(Definitions related to alcove walks will be given in Section \ref{section: alcove walks}.) The coefficients $a^\enn(p)\in\mathbb{F}^\enn$ are given explicitly in Theorem~\ref{main}.
\end{customthm}
In Theorem~\ref{permuted basement}, we record an alcove walk formula for $T_u E_\mu^\enn$ (where $u\in W_0$ and $\mu\nolinebreak\in\nolinebreak\mathbb{Z}^r$). In the case $n=1$, this recovers the alcove walk formula of \cite{ry}. The polynomials $T_u E_\mu=T_u E_\mu^{(1)}$ were given the name \emph{permuted basement Macdonald polynomials} in \cite{permutedBasement}. This name alludes to the remarkable combinatorial formula satisfied by $T_u E_\mu$, stated in \cite{permutedBasement} and proven in \cite{permutedBasement2}, which generalizes the well-known formula of \cite{hhl} for $E_\mu$.
The polynomials $T_u E_\mu$ are also studied in \cite{cmw, fm,fmo,os}. 

For dominant $\mu\in\mathbb{Z}^r$, we will define the symmetric SSV polynomial $P_\mu^\enn$, which is symmetric with respect to a conjugate of the Chinta-Gunnells Weyl group action: see \cite{cg1,cg2}. (Note that in general, $P_\mu^\enn$ is \emph{not} symmetric with respect to the usual action of the Weyl group.) When $n=1$, $P_\mu^\enn$ recovers the symmetric Macdonald polynomial $P_\mu$.  
In Theorem~\ref{thm: symmetric}, we obtain the following alcove walk formula for $P_\mu^\enn$:
\begin{customthm}{B}\label{intro symmetric}
Let $\mu\in\mathbb{Z}^r$ be dominant. For the same $\lambda\in\mathbb{Z}^r$ and affine Weyl group element $w$ as in the previous result, we have
\begin{equation}\label{intro P formula}
    P_\mu^\enn = \sum_{u\in S_r}k^{\ell(u)}\sum_{p\in\mathcal{B}(u,\vec{w})} b^\enn(p)x^{n\textnormal{wt}(p)+\phi(p)\lambda},
\end{equation}
where $\mathcal{B}(u,\vec{w})$ is the set of all alcove walks of type $(i_1,\dots,i_\ell)$ starting at $u$ and the coefficients $b^\enn(p)\in\mathbb{F}^\enn$ are given explicitly in Theorem~\ref{thm: symmetric}. 
\end{customthm}

In Theorems~\ref{main q limit} and \ref{symmetric q limit}, we give the limits of (\ref{intro E formula}) and (\ref{intro P formula}) as $q$ approaches $0$ and $\infty$. The two limits can be expressed in terms of alcove walks with only positive and only negative folds respectively. We also observe that the coefficients are Laurent polynomials in the metaplectic parameters, and with suitable assumptions on $k$, they have nonnegative coefficients. These limits are of particular interest because the $q\rightarrow\infty$ limit of (\ref{intro E formula}) is related to metaplectic Iwahori-Whittaker functions: see the upcoming work \cite{ssv2}. 

As an application of Theorem~\ref{intro main}, we derive the following result, which is Corollary~\ref{triangularity corollary} below. The triangularity result will also appear in \cite{ssv2}, where it is proven without the use of an explicit formula for $E_\mu^\enn$. 
\begin{customcor}{C}\label{intro triangularity}
For $\mu\in\mathbb{Z}^r$, 
\begin{equation*}
    E_\mu^\enn =x^\mu+\sum_{\nu\overset{ n}{< }\mu} c_\nu x^{\nu}
\end{equation*}
for some \emph{nonzero} scalars $c_\nu\in\mathbb{F}^\enn$. (See the discussion following Theorem~\ref{main} for the definition of the partial order $\overset{ n}{< }$.)  
\end{customcor}

This result allows us to precisely characterize the powers of $x$ appearing in $E_\mu^\enn$. In the case $n=1$, the fact that the coefficients $c_\nu$ are nonzero is known: see \cite{sahiUnpublished}, or \cite{sahiKoornwinder2} for the case of Koornwinder polynomials.

We may also relate the terms of SSV polynomials to the terms of the corresponding Macdonald polynomials. This result is stated in greater generality in Corollary~\ref{macdonald comparison} below. 

\begin{customcor}{D}\label{intro macdonald comparison}
Let $\mu,\nu\in\mathbb{Z}^r$. If $x^\nu$ appears with nonzero coefficient in $E_\mu^\enn$, then $x^\nu$ appears with nonzero coefficient in the Macdonald polynomial $E_\mu^{(1)}$.
\end{customcor}

The structure of the work is as follows. In Section~\ref{basics}, we give all the necessary terminology for working with root systems, Weyl groups, alcove walks, and double affine Hecke algebras. We begin with the standard notions, then introduce ``metaplectic" variants that depend on the positive integer $n$. We also give formulas for the representation $\pi^\enn$ constructed in \cite{ssv} and obtain some basic consequences. 
In Section~\ref{section: ssv polys}, we construct the SSV polynomials and the intertwiners, then use them to derive the main combinatorial formula, Theorem~\ref{intro main}. We apply this result to prove Corollary~\ref{intro triangularity}.
We also prove that for fixed $\mu\in\mathbb{Z}^r$ and positive integers $m|n$, the powers of $x$ appearing in our formula (\ref{intro E formula}) for $E_\mu^\enn$ are a subset of the powers appearing in the corresponding formula for $E_\mu^{(m)}$. Corollary~\ref{intro macdonald comparison} follows from this result.  
In Section~\ref{section: symmetric}, we define and establish some basic properties of symmetrized SSV polynomials $P_\mu^\enn$ analogous to those for the symmetric Macdonald polynomials, then prove Theorem~\ref{intro symmetric}. In Section~\ref{section: q limits} give alcove walk formulas for $q$-limits of $E_\mu^\enn$ and $P_\mu^\enn$ and observe a simple positivity result. We conclude with an appendix listing, for several small values of $\mu$ and $n$, the expansions of $E_\mu^\enn$ and $P_\mu^\enn$ according to the new formulas. 

\section*{Acknowledgements}
The author would like to thank Siddhartha Sahi for his advice regarding this work and his mentorship throughout the author's PhD program. He would also like to thank Songhao Zhu for his collaboration during the early stages of this project; Vidya Venkateswaran for her insightful talks at Rutgers University explaining the content of \cite{ssv}; and Henrik Gustafsson, Jasper Stokman, and Vidya Venkateswaran for noticing mistakes in and suggesting improvements to this work. 

\section{Background and notation}\label{basics}

\subsection{Root Systems and Weyl Groups}\label{section: root systems}

Fix $r\geq 2$. Let $\{\epsilon_i: 1\leq i\leq r\}$ be the standard orthonormal basis of $\mathbb{R}^r$, with associated inner product $(\cdot,\cdot)$ and norm $||\cdot||$. The root system of type $A_{r-1}$ is 
\begin{equation*}
    \Phi=\{\epsilon_i-\epsilon_j: 1\leq i\neq j\leq r\}.
\end{equation*}
The set of positive roots is 
\begin{equation*}
    \Phi_+=\{\epsilon_i-\epsilon_j: 1\leq i< j\leq r\},
\end{equation*}
and the set of simple roots is 
\begin{equation*}
    \Delta=\{\alpha_1, \dots, \alpha_{r-1}\},
\end{equation*}
where $\alpha_i=\epsilon_i-\epsilon_{i+1}$. The highest root is 
\begin{equation}\label{def: theta}
    \theta=\epsilon_1-\epsilon_r=\alpha_1+\dots+\alpha_{r-1}.
\end{equation}
The root lattice is $Q=\textnormal{span}_{\mathbb{Z}}\Phi$, and the GL$_r$ weight lattice is $P=\textnormal{span}_{\mathbb{Z}}\{\epsilon_1, \dots,\epsilon_r\}$, which we identify with $\mathbb{Z}^r.$ (We will refer to this space as both $P$ and $\mathbb{Z}^r$.) We will use the notations 
\begin{equation*}
    \lambda=\sum_{i=1}^r \lambda_i \epsilon_i=(\lambda_1, \dots, \lambda_r)
\end{equation*}
interchangeably for an element of $P\cong\mathbb{Z}^r$. $\lambda\in P$ is called \emph{dominant} if, for all $1\leq i<r$, $\lambda_i\geq \lambda_{i+1}$. 

The Weyl group of the finite root system of type $A_{r-1}$ is 
\begin{equation*}
    W_0= S_r,
\end{equation*}
the symmetric group on $\{1, \dots, r\}$. 
For $1\leq i\leq r-1$, let $s_i$ be the transposition $(i\hspace{.1 in} i+1)$. It is well-known that $s_1,\dots,s_{r-1}$ generate $W_0$. 

For any root $\alpha\in \Phi$, we have the reflection $s_\alpha$, the transformation of $\mathbb{R}^r$ given by
\begin{equation}\label{def: refl}
    s_\alpha\lambda=\lambda-(\lambda,\alpha)\alpha
\end{equation}
(Note that we do not need to discuss coroots here, since $(\alpha,\alpha)=2$ for all $\alpha\in \Phi$.) Explicitly, $s_{\alpha_i}$ is the transformation that exchanges $\lambda_i$ and $\lambda_{i+1}$. Note that $P$ is invariant under these transformations. 
We also have 
\begin{equation*}
    s_\theta=s_1 s_2\cdots s_{r-2}s_{r-1}s_{r-2}\cdots s_2 s_1.
\end{equation*}

There is an action of $W_0$ on $\mathbb{R}^r$, and on $\mathbb{Z}^r\subseteq \mathbb{R}^r$, given by
\[s_i\mapsto s_{\alpha_i}.\]
Under this map, the element $s_\theta\in W_0$ is taken to the reflection
$s_\theta$ of $\mathbb{R}^r$ (recall the definition of $\theta$ in (\ref{def: theta})), so there is no conflict of notation. 


Consider the space $\mathbb{R}^r_a=\mathbb{R}^{r+1}$. We identify the subspace
\[\{(v_1,\dots,v_r,0): v_1,\dots,v_r\in\mathbb{R}\}\]
with $\mathbb{R}^r$ and let $\delta=(0,\dots,0,1)$, so that
\[\mathbb{R}^r_a=\mathbb{R}^r\oplus \mathbb{R}\delta.\]
We then define 
\begin{equation*}
    P_a=P\oplus \mathbb{Z}\delta\cong \mathbb{Z}^r\oplus \mathbb{Z}\delta.
\end{equation*}
The symmetric bilinear form $(\cdot,\cdot)$ on $\mathbb{R}^r$ extends to a symmetric bilinear form on $\mathbb{R}^r_a$ (and its subspace $P_a$) by defining
\[(\delta,v+s\delta)=0\] 
for all $v\in P$ and $s\in\mathbb{R}$. 

The 
set of (real) affine roots is 
\begin{equation*}
    \widetilde{\Phi}=\{\alpha+s\delta: \alpha\in \Phi, s\in\mathbb{Z}\}\subseteq P_a.
\end{equation*}
Define 
\begin{equation*}
    \alpha_0=\delta-\theta\in \widetilde{\Phi}.
\end{equation*}
The set of (real) \emph{positive affine roots} is
\begin{equation*}
    \widetilde{\Phi}_+=\Phi_+\cup \{\alpha+s\delta: s>0, \alpha\in \Phi\}.
\end{equation*}
Every element of $\widetilde{\Phi}_+$ is a nonnegative linear combination of $\alpha_0,\dots,\alpha_{r-1}$.

We have the extended affine Weyl group $W=P\rtimes W_0$. We denote the elements of $W$ by $\tau(\mu)w$ for $\mu\in P$ and $w\in W_0=S_r$. They satisfy the relation
\begin{equation*}
    w\tau(\mu)w^{-1}=\tau(w\mu).
\end{equation*}
Let 
\begin{equation*}
    s_0=\tau(\theta) s_\theta
\end{equation*}
and 
\begin{equation*}
    \omega=s_1 s_2 \cdots s_{r-1}\tau(\epsilon_r).
\end{equation*}
Then $W$ is generated by $s_0,\dots,s_{r-1},\omega$. Let
\begin{equation*}
    W_{\textnormal{Cox}}=Q\rtimes W_0\subseteq W.
\end{equation*}
$W_{\textnormal{Cox}}$ is a Coxeter group with generators $s_0,\dots, s_{r-1}$. For $w\in W_{\textnormal{Cox}}$, the \emph{length} of $w$, denoted $\ell(w)$, is the number of generators in a reduced expression for $w$. 

We have a representation of $W$ on $\mathbb{R}^r_a$, where
\begin{equation*}
    s_i*\widehat{v}=\widehat{v}-(\widehat{v},\alpha_i)\alpha_i=(s_iv)+s \delta
\end{equation*}
and 
\begin{equation*}
    \tau(\mu)*\widehat{v}=\widehat{v}-(\widehat{v},\mu)\delta=v+(s-(v,\mu))\delta.
\end{equation*}
for $1\leq i\leq r-1$, $\mu\in P$, and $\widehat{v}=v+s\delta\in \mathbb{R}^r_a$. 
Note that this restricts to a linear action on $P_a$. 

We also have an affine action of $W$ on $\mathbb{R}^r$, where $s_i\in W_0$ acts on $v\in\mathbb{R}^r$ by the reflection $s_{\alpha_i}$ as in (\ref{def: refl}), and 
\begin{equation*}
    \tau(\mu)v=v+\mu
\end{equation*}
for $\mu\in P$. It is easy to see directly that this action is faithful. 
Under the affine action, we have
\begin{equation*}
    s_0 (v_1, \dots, v_r)=(v_r+1, v_2, v_3,\dots, v_{r-1}, v_1-1).
\end{equation*}
Two invariant subsets of particular interest are $\mathbb{Z}^r$ and
\begin{equation*}
    \mathfrak{h}^*=\textnormal{span}_{\mathbb{R}}\Phi=\{(v_1, \dots, v_r)\in\mathbb{R}^r: \sum v_i=0.\}
\end{equation*}

Notice that for $w\in W$, we denote the linear action on $v\in\mathbb{R}^r_a$ by $w*v$ and the affine action on $\widehat{v}\in\mathbb{R}^r$ by $w\widehat{v}$. For $w\in W_0$, where the two actions are essentially the same, we will drop the $*$ from our notation. 

\begin{remark}
Both the affine and linear actions of $W$ come from the action of $W$ on the Cartan subalgebra of the corresponding untwisted affine Lie algebra. Details may be found in Sections 6.5 and 6.6 of \cite{kac}. 
\end{remark}

We view $P_a$ as a set of affine functions on $\mathbb{R}^r$ by 
\begin{equation}\label{angle bracket n}
    \langle \mu+s\delta,v\rangle =(\mu,v)+s,
\end{equation}
where $\mu\in P$, $s\in\mathbb{Z}$, and $v\in \mathbb{R}^r$. 
We may then define, for all $\widehat{\alpha}=\alpha+s\delta\in P_a$ with $(\alpha,\alpha)\neq 0$, the affine transformation 
\[s_{\widehat{\alpha}}:\mathbb{R}^r\rightarrow \mathbb{R}^r\] 
given by
\begin{equation}\label{general reflection}
    s_{\widehat{\alpha}}v=v-\dfrac{2\langle \widehat{\alpha},v\rangle}{(\alpha,\alpha)} \alpha
\end{equation}
for $v\in\mathbb{R}^r$. In particular, if $\widehat{\alpha}=m\beta+sm^2\delta$ for some positive integer $m$, then
\begin{equation}\label{n reflection}
    s_{\widehat{\alpha}}v=v-\langle \beta+sm\delta,v\rangle\beta=(s_\beta v)-sm\beta.
\end{equation}
Of course, for $0\leq i\leq r-1$ and $v\in \mathbb{R}^r$, this reduces to
\begin{equation}\label{si as reflection}
    s_{\alpha_i}v=s_i v.
\end{equation}
Further, for all $\widehat{\alpha}\in \widetilde{\Phi}$, there is a unique element $w\in W$ such that for all $v\in\mathbb{R}^r$, 
\[wv=s_{\widehat{\alpha}}v.\]
We will identify $w$ and $s_{\widehat{\alpha}}$.

\subsection{Alcove walks and Bruhat order}\label{section: alcove walks}

We now give the necessary definitions to work with alcove walks and the Bruhat order. We will mostly follow the notation of \cite{ry} for alcove walks.

For $\widehat{\alpha}\in\widetilde{\Phi}$, let
\begin{equation*}
    \mathfrak{h}^{\widehat{\alpha}}=\{v\in\mathbb{R}^r: \langle \widehat{\alpha},v\rangle =0\}
\end{equation*}
(recalling the notation (\ref{angle bracket n})). 
We then define the \emph{alcoves} of $\mathbb{R}^r$ to be the connected components of 
\begin{equation*}
    \mathbb{R}^r\setminus\left( \bigcup_{\widehat{\alpha}\in \widetilde{\Phi}} \mathfrak{h}^{\widehat{\alpha}}  \right).
\end{equation*}
We call the hyperplanes bounding an alcove its \emph{walls}. 
The \emph{fundamental alcove} is 
\begin{equation}\label{fundamental alcove}
    \mathcal{A}=\{v\in\mathbb{R}^r: \langle \alpha_i,v\rangle >0\textnormal{ for }0\leq i\leq r-1\}=\{(v_1, \dots, v_r)\in\mathbb{R}^r: v_1 > v_2 > \dots > v_r, v_1-v_r<1\},
\end{equation}
with walls $\mathfrak{h}^{\alpha_0},\dots,\mathfrak{h}^{\alpha_{r-1}}$. (Note that the fundamental alcove is determined by our choice of simple roots.)
The action of $W_{\textnormal{Cox}}$ sends alcoves to alcoves: in fact, for $0\leq i\leq r-1$, $s_i$ acts by reflecting in the wall $\mathfrak{h}^{\alpha_{i}}$.
 We also have the \emph{closed fundamental alcove}
\begin{equation*}
    \overline{\mathcal{A}}=\{v\in\mathbb{R}^r: \langle \alpha_i,v\rangle \geq 0\textnormal{ for }0\leq i\leq r-1\}=\{(v_1, \dots, v_r)\in\mathbb{R}^r: v_1 \geq v_2 \geq \dots > v_r, v_1-v_r\leq 1\}.
\end{equation*}

The following lemma is well-known: see, for instance, \cite{humphCox}. 

\begin{lemma}\label{fundamental domain R}
For any $v\in \mathbb{R}^r$, there exist $w\in W_{\textnormal{Cox}}$ and unique $v_+\in \overline{\mathcal{A}}$ such that $wv_+=v$. If we require $w$ to have the shortest possible length, then $w$ is also unique. 
\end{lemma}

Since the fundamental alcove $\mathcal{A}$ is not preserved by the action of any nontrivial element of $W_{\textnormal{Cox}}$, this lemma gives a bijection between $W_{\textnormal{Cox}}$ and the set of alcoves in $\mathbb{R}^r$, given by
\begin{equation*}
    w\mapsto w\cdot \mathcal{A}
\end{equation*}
We identify $W_{\textnormal{Cox}}$ and the set of alcoves via this bijection. Thus we will sometimes use the affine Weyl group element $1$ to refer to the fundamental alcove $\mathcal{A}$, as in \cite{ry}.

\begin{figure}[ht]
\begin{center}
\begin{tikzpicture}
\node[inner sep=0pt] (mainpic) at (0,0)
   {\includegraphics[scale=0.8]{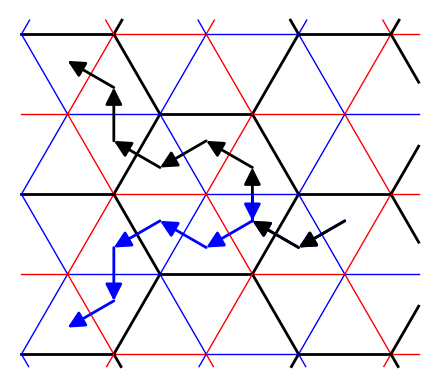}};
\node[draw=none] at (4.5,0.0) {$\mathfrak{h}^{\alpha_0}$};
\node[draw=none] at (3.85,0.15) {\scriptsize $+$};
\node[draw=none] at (3.85,-0.15) {\scriptsize $-$};

\node[draw=none] at (4.5,1.0) {$\mathfrak{h}^{\alpha_2}$};
\node[draw=none] at (3.85,1.0) {\scriptsize $+$};
\node[draw=none] at (4.1,0.75) {\scriptsize $-$};

\node[draw=none] at (-0.6,3.7) {$\mathfrak{h}^{\alpha_1}$};
\node[draw=none] at (-0.3,3.6) {\scriptsize $+$};
\node[draw=none] at (-0.6,3.4) {\scriptsize $-$};

\node[draw=none] at (2.8,-0.35) {$\mathcal{A}$};
\node[draw=none] at (-3.1,3) {\footnotesize end$(p)$};
\node[draw=none] at (-3.1,-3) {\footnotesize end$(p')$};
\end{tikzpicture}
\end{center}
\caption{\small Alcoves and alcove walks in the case $r=3$. We restrict our attention to $\textnormal{span}_{\mathbb{R}}\{\alpha_1,\alpha_2\}$ so that our picture is two-dimensional. We see two alcove walks of type $(1,0,1,2,1,0,1,2)$ starting at the fundamental alcove $\mathcal{A}$: $p$ (black) and $p'$ (blue once it deviates from $p$). $p$ is unfolded, and $p'$ is obtained by folding $p$ at step $3$; this fold is negative. $p$ has length $8$, and $p'$ has length $7$.} 
\label{figure}
\end{figure}
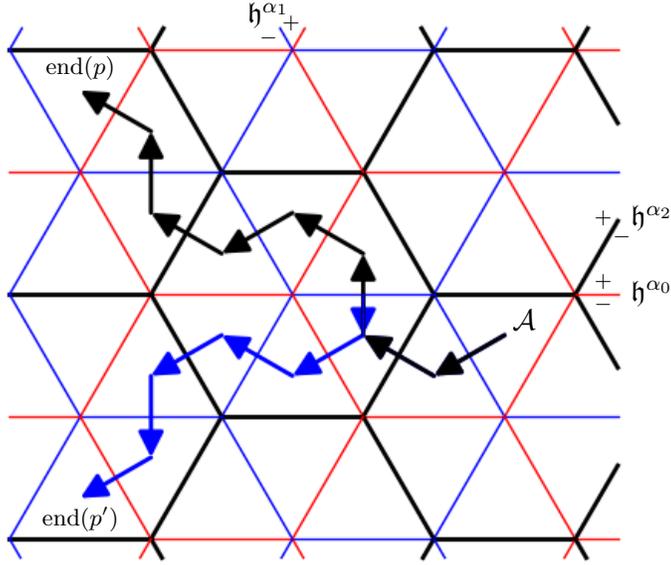

There is a natural orientation on the set of hyperplanes, determined by the choice of positive roots, defined as follows. Consider a hyperplane $\mathfrak{h}^{\widehat{\alpha}}$ where
$\widehat{\alpha}\in\widetilde{\Phi}$ has the form
\begin{equation*}
    \widehat{\alpha}=\alpha+s\delta, \alpha\in\Phi_+.
\end{equation*}
For $v\in\mathbb{R}^r$, we say that $v$ is on the positive side of $\mathfrak{h}^{\widehat{\alpha}}$ if $\langle \widehat{\alpha},v\rangle>0$ and is on the negative side of $\mathfrak{h}^{\widehat{\alpha}}$ if $\langle \widehat{\alpha},v\rangle<0$. (Note that points in $\mathcal{A}$ are on the positive side of $\mathfrak{h}^{\alpha_i}$ for $1\leq i\leq r-1$ but are on the \emph{negative} side of $\mathfrak{h}^{\alpha_0}=\mathfrak{h}^{\theta-\delta}$.) This is the ``periodic orientation" described in \cite{ry}, so called because for any $\alpha\in \Phi$ and $s,t\in\mathbb{Z}$, $\mathfrak{h}^{\alpha+s\delta}$ and $\mathfrak{h}^{\alpha+t\delta}$ have parallel orientations. 

For $w\in W_{\textnormal{Cox}}$, fix a reduced expression $w=s_{i_1}\cdots s_{i_\ell}$. Let $\vec{w}=(i_1, \dots, i_\ell)$. An \emph{alcove walk} of type $\vec{w}$ is a sequence 
\begin{equation}\label{alcove walk}
    p=(A_0, A_1, \dots, A_\ell)
\end{equation}
of alcoves, where for each $1\leq j\leq \ell$, either $A_j=A_{j-1}$ or $A_j=A_{j-1}s_{i_j}$. (Recalling the identification of alcoves and elements of $W_{\textnormal{Cox}}$, the notation $A_{j-1}s_{i_j}$ refers to the product in $W_{\textnormal{Cox}}$. The corresponding alcove shares a wall with the alcove $A_{j-1}$.) 
We say that $p$ has a \emph{fold} at step $j$ if $A_j=A_{j-1}$ and a \emph{crossing} otherwise: this fold or crossing is \emph{positive} if $A_{j-1}$ is on the positive side of the hyperplane separating $A_{j-1}$ and $A_{j-1}s_{i_j}$, and is negative otherwise. Let $f_+(p)$ (respectively $f_-(p)$) be the set of all $j$ such that $p$ has a positive (respectively negative) fold at step $j$. 
We say $p$ is \emph{unfolded} if it has no folds. 
We define the \emph{length} of $p$ to be the number of hyperplanes of the form $\mathfrak{h}^{\widehat{\alpha}}$ between $A_0$ and $A_\ell$. If $p$ is unfolded, then the length of $p$ is $\ell$.
Examples of these notions are given in Figure~\ref{figure}.

For $u\in W_{\textnormal{Cox}}$ and $\vec{w}$ as above, define $\mathcal{B}(u,\vec{w})$ to be the set of all alcove walks of type $\vec{w}$ with $A_0=u$. We use the notation $\mathcal{B}(\vec{w})=\mathcal{B}(u,\vec{w})$. 

Given an alcove walk $p$ as in (\ref{alcove walk}), let end$(p)=A_\ell$. We have the decomposition 
\begin{equation}\label{end, weight, phi}
    \textnormal{end}(p)=\tau(\textnormal{wt}(p))\phi(p),
\end{equation}
where wt$(p)\in Q$ (since we are only using $W_{\textnormal{Cox}}$, not $W$) and $\phi(p)\in W_0$. Suppose further that $p$ does not have a fold at step $j$, and the hyperplane separating the alcoves $A_{j-1}$ and $A_j$ is $\mathfrak{h}^{\widehat{\alpha}}$ for $\widehat{\alpha}\in\widetilde{\Phi}_+$. We say that an alcove walk $p'=(A_0',\dots, A_\ell')$ of the same type as $p$ is obtained by \emph{folding $p$ at step $j$} if $A_i'=A_i$ for $i<j$, $A_j'=A_{j-1}$, and $A_i'=s_{\widehat{\alpha}}A_i$ for $i>j$. Geometrically, to obtain $p'$ from $p$, we reflect the tail of $p$ in the hyperplane $\mathfrak{h}^{\widehat{\alpha}}$. Note that
\begin{equation}\label{folding ends}
    \textnormal{end}(p')=s_{\widehat{\alpha}}\textnormal{end}(p).
\end{equation}

\begin{remark}\label{alcove walk explanation}
For the unfamiliar, we would like to elaborate on the relationship between words in $W_{\textnormal{Cox}}$ and alcove walks. Let $w\in W_{\textnormal{Cox}}$ and consider the corresponding alcove, also called $w$. Given an expression $w=s_{i_1}\cdots s_{i_\ell}$, we may use the property $s_j s_{\widehat{\alpha}}s_j=s_{s_j * \widehat{\alpha}}$ to write
\begin{equation*}
    w=s_{i_1}\cdots s_{i_\ell}= s_{s_{i_{\ell -1}}\cdots s_{i_1}*\alpha_{i_\ell}} \cdots s_{s_{i_1}*\alpha_{i_2}}s_{\alpha_{i_1}}.
\end{equation*}
We note that each alcove in the sequence
\begin{equation*}
    \mathcal{A}, s_{\alpha_{i_1}}\mathcal{A}, s_{s_{i_1}*\alpha_{i_2}}s_{\alpha_{i_1}}\mathcal{A}, \dots, s_{s_{i_{\ell -1}}\cdots s_{i_2}*\alpha_{i_\ell}} \cdots s_{s_{i_1}*\alpha_{i_2}}s_{\alpha_{i_1}}\mathcal{A}=w
\end{equation*}
shares a wall with the one preceding it: for example, the wall between $s_{\alpha_{i_1}}\mathcal{A}$ and $s_{s_{i_1}*\alpha_{i_2}}s_{\alpha_{i_1}}\mathcal{A}$ is the hyperplane corresponding to $s_{i_1}*\alpha_{i_2}$, and the same relationship holds in general. 
This gives an unfolded alcove walk with starting alcove $\mathcal{A}=1$ and ending alcove $w$. In fact, this yields a bijection between expressions for $w$ and unfolded alcove walks (of any type) that start at $1$ and end at $w$. Under this bijection, reduced expressions for $w$ correspond to unfolded alcove walks from $1$ to $w$ \emph{with the smallest possible number of steps}. This allows us to characterize the length of $w$ as the length of an unfolded alcove walk from $1$ to $w$.
%

For $w=s_{i_1}\cdots s_{i_\ell}$, a reduced expression, we call expressions 
\begin{equation*}
    s_{i_{j_1}}\cdots s_{i_{j_b}}
\end{equation*}
for $1\leq j_1<\cdots <j_b\leq \ell$ \emph{subwords} of $s_{i_1}\cdots s_{i_\ell}$. Given an alcove walk $p$ of type $(i_1, \dots, i_\ell)$, let $1\leq j_1<\dots<j_b\leq \ell$ be the indices at which $p$ does \emph{not} have a fold. Then 
\begin{equation*}
    \textnormal{end}(p)=s_{i_{j_1}}\cdots s_{i_{j_b}}.
\end{equation*}
\end{remark}

Finally, we recall the Bruhat order. Let $\overline{W}$ be any Coxeter group and let $\overline{\Phi}$ be its root system. We will use the following definition and lemma in one of two contexts: $\overline{W}=W_{\textnormal{Cox}}$ and $\overline{\Phi}=\widetilde{\Phi}$, or $\overline{W}=W_{\textnormal{Cox}}^\enn$ and $\overline{\Phi}=\widetilde{\Phi}^\enn$ (see Section~\ref{section: metaplectic}). 

\begin{definition}\label{bruhat def}
For $w\in \overline{W}$ and $\widehat{\alpha}\in \overline{\Phi}$, define $w<s_{\widehat{\alpha}}w$ if $\ell(w)<\ell(s_{\widehat{\alpha}}w)$. Define a partial order $\leq$, which we call the \emph{Bruhat order}, on $\overline{W}$ by taking the transitive closure of this relation. 
\end{definition}

The following result may be found in \cite{humphCox} for a general Coxeter group. 

\begin{lemma}\label{bruhat subwords}
Let $w,w'\in \overline{W}$ and fix a particular choice of reduced expression $w=s_{i_1}\cdots s_{i_\ell}$ for $w$. $w'\leq w$ if and only if $w'$ can be obtained as a subword of this reduced expression. 
\end{lemma}

\subsection{Metaplectic versions}\label{section: metaplectic}

For the rest of the paper, fix a positive integer $n$, which we call the \emph{metaplectic degree}. For any integer $j$, let $r_n(j)\in\{0, \dots, n-1\}$ be the residue modulo $n$. 
Fix an invertible parameter $k$. 
Also fix parameters $G^\enn_j$ for $j\in\mathbb{Z}$, satisfying the following conditions:
\begin{enumerate}
    \item $G^\enn_0=k.$
    \item $G^\enn _j G^\enn_{-j}=1.$
    \item $G_{j}^\enn=G_{r_n(j)}^\enn$.
\end{enumerate}
\begin{remark}
These conditions imply that $\{G^\enn_j:j\in\mathbb{Z}\}$ is determined by $G^\enn_0, \dots, G^\enn_{\lfloor \frac{n}{2}\rfloor}$. Further, if $n$ is even, we have
\begin{equation*}
    G_{\frac{n}{2}}^\enn=\pm 1.
\end{equation*}
In \cite{ssv}, parameters $g_j^\enn$ are used, with the dictionary being
\begin{equation*}
    G^\enn_j=-kg^\enn_{-j}.
\end{equation*}
In that work, the scalar $G_{\frac{n}{2}}^\enn\in\{\pm 1\}$ is called $\epsilon$. We will continue to call it $G_{\frac{n}{2}}^\enn$ in order to keep the notation uniform, but will keep in mind that it is not a free parameter. 
\end{remark}

We introduce the fields
\begin{equation*}
    \mathbb{K}^{(n)}=\mathbb{C}(k,G^{(n)}_1, \dots, G^{(n)}_{\lfloor \frac{n-1}{2}\rfloor})\textnormal{ and }\mathbb{F}^\enn = \mathbb{K}^\enn (q),
\end{equation*}
where $q$ is another independent parameter. 
In particular, $\mathbb{K}^{(1)}=\mathbb{K}^{(2)}=\mathbb{C}(k)$ and $\mathbb{K}^{(3)}=\mathbb{C}(k,G_1^{(3)}).$

We now introduce versions of the concepts in the previous sections that depend on the metaplectic degree $n$. The case $n=1$ will recover the usual notions. Let 
\begin{equation*}
    W^\enn=nP\rtimes W_0\subseteq W.
\end{equation*}
In terms of the affine action on $\mathbb{R}^r_a$, $W^\enn$ is generated by reflections in the finite roots and translations by $nP$. We have an isomorphism $\Psi^\enn: W\rightarrow W^\enn$ given by
\begin{equation*}
    s_i\mapsto s_i\textnormal{ and }\tau(\lambda)\mapsto\tau(n\lambda)
\end{equation*}
for $1\leq i\leq r-1$ and $\lambda\in P$. 

\begin{remark}
Since $W$, and therefore its subgroup $W^\enn$, acts faithfully on $\mathbb{R}^r$, we may identify both with groups of affine transformations of $\mathbb{R}^r$. In this context, it is clear that $\Psi^\enn$ is an isomorphism, since the two groups differ only by scaling. 
\end{remark}

The ``metaplectic" versions of the other notions are all determined by the map $\Psi^\enn$. Let 
\begin{equation*}
    s_0^\enn=\Psi^\enn(s_0)=\tau(n\theta)s_\theta,
\end{equation*}
and for uniformity of notation let 
\begin{equation*}
    s_i^\enn=\Psi^\enn(s_i)=s_i
\end{equation*}
for $1\leq i\leq r-1$. Also define
\begin{equation}\label{omega def}
    \omega^\enn=\Psi^\enn(\omega)=s_1 s_2\cdots s_{r-1} \tau(n\epsilon_r).
\end{equation}
Then $s_0^\enn, \dots, s_{r-1}^\enn,\omega^\enn$ generate $W^\enn$. We also introduce the group 
\begin{equation*}
    W^\enn_{\textnormal{Cox}}=\Psi^\enn(W_{\textnormal{Cox}})=nQ\rtimes W_0.
\end{equation*}
This is a Coxeter group with generators $s_0^\enn, \dots, s_{r-1}^\enn$.

Explicitly, for $v\in \mathbb{R}^r$, we have
\begin{equation*}
    s_0^\enn (v_1, \dots, v_r)=(v_r+n, v_2, \dots, v_{r-1}, v_1-n).
\end{equation*}
Identifying $W^\enn$ with the corresponding group of affine transformations as above and recalling (\ref{general reflection}) and (\ref{n reflection}), we see that for $v\in \mathbb{R}^r$,
\begin{equation*}
    s_{-n\theta+n^2 \delta }v=(s_\theta v)+n\theta=s_0^\enn v,
\end{equation*}
so
\begin{equation*}
    s_0^\enn=s_{-n\theta+n^2 \delta}.
\end{equation*}

This motivates the definition of the space
\begin{equation*}
    P_a^\enn=nP\oplus n^2\mathbb{Z}\delta\subseteq P_a.
\end{equation*}
and the isomorphism of abelian groups (but \emph{not} of lattices, since $(\cdot,\cdot)$ is not preserved)
\begin{equation*}
    \Psi^\enn: P_a\rightarrow P_a^\enn 
\end{equation*}
defined by
\begin{equation*}
    \Psi^\enn(\lambda+s\delta)=n\lambda+sn^2\delta.
\end{equation*}
With these definitions, and $\alpha_i^\enn$ defined for $0\leq i\leq r-1$ by 
\begin{equation*}
    \alpha_i^{(n)}=\Psi^\enn(\alpha_i)=\left\{
     \begin{array}{lr}
       n\alpha_i & : i\neq 0\\
       -n\theta + n^2 \delta & : i=0,
     \end{array}
   \right.
\end{equation*}
we have
\begin{equation*}
    s_i^\enn=s_{\alpha_i^\enn}.
\end{equation*}
for $0\leq i\leq r-1$. In addition, for $\lambda\in P$, 
\begin{equation*}
    \Psi^\enn(\tau(\lambda))=\tau(\Psi^\enn(\lambda)).
\end{equation*}

The affine root system for which $\alpha_0^\enn, \dots, \alpha_{r-1}^\enn$ are the simple roots is
\begin{equation*}
    \widetilde{\Phi}^\enn=\Psi^\enn(\widetilde{\Phi})=\{n\alpha+sn^2\delta: \alpha\in \Phi\cup\{0\}, s\in\mathbb{Z},\textnormal{ and }\alpha\neq 0\textnormal{ or }s\neq 0\}\subseteq P_a^\enn.
\end{equation*}
(Note that the simple roots $\alpha_i^\enn$ were called $b_i^\enn$ in \cite{ssv}.) 
The corresponding set of (real) positive affine roots is
\begin{equation*}
    \widetilde{\Phi}_+^\enn=\Psi^\enn(\widetilde{\Phi}_+)=n\Phi_+\cup \{n\alpha+sn^2\delta: s>0, \alpha\in \Phi\cup\{0\}\}.
\end{equation*}

\begin{remark}
We intentionally use the same notation for the maps
\[\Psi^\enn: W\rightarrow W^\enn\textnormal{ and }\Psi^\enn: P_a\rightarrow P_a^\enn,\]
since one is induced by the other. We will use the same notation in Proposition~\ref{psi hecke} for a related isomorphism of Hecke algebras. The context should always make it clear which map is being applied. 
\end{remark}

We would also like to record the following property, which is easy to check: for all $\widehat{\mu}\in P_a$, $v\in\mathbb{R}^r$, and $0\leq i\leq r-1$,
\begin{equation}\label{self adjoint}
    \langle s_i^\enn * \widehat{\mu},v\rangle =\langle \widehat{\mu},s_i^\enn v\rangle.
\end{equation}

Define
\begin{equation*}
    A^\enn=n(\mathcal{A}\cap \frac{1}{n}\mathbb{Z}^r)=\{(\lambda_1, \dots, \lambda_r)\in \mathbb{Z}^r: \lambda_1 \geq \lambda_2 \geq \dots \geq \lambda_r, \lambda_1-\lambda_r\leq n\}.
\end{equation*}

For $v\in \mathbb{R}^r$, we will use the notation
\begin{equation*}
    \Psi^\enn(v)=nv.
\end{equation*}
Then for $w\in W_{\textnormal{Cox}}$ and $v\in\mathbb{R}^r$, we observe that
\begin{equation}\label{psi compatibility}
    \Psi^\enn(w)\Psi^\enn(v)=\Psi^\enn(wv).
\end{equation}
This allows us to identify the action of $W_{\textnormal{Cox}}$ on $\frac{1}{n}\mathbb{Z}^r\subseteq \mathbb{R}^r$ with the action of $W_{\textnormal{Cox}}^\enn$ on $\mathbb{Z}^r$. Then Lemma~\ref{fundamental domain R} has the following consequence.
 
\begin{lemma}\label{lambda and w}
For any $\mu\in \mathbb{Z}^r$, there exist $w\in W_{\textnormal{Cox}}^\enn$ and a unique $\lambda\in A^\enn$ such that $w\lambda=\mu$. If we require $w$ to have the shortest possible length, then $w$ is unique. 
\end{lemma}

\subsection{Hecke algebras}

We now define two closely related Hecke algebras and study the representation given in \cite{ssv}. The first algebra is the object of study in \cite{ry} with a slightly enlarged center. Recall that we have fixed the positive integer $n$. 

\begin{definition}\label{daha}
The \emph{double affine Hecke algebra} $\mathbb{H}$ is the unital associative algebra over $\mathbb{K}^\enn$ generated by $q_n,\omega,\omega^{-1},T_0, T_1, \dots, T_{r-1},$ and $X^{\widehat{\mu}}$ for $\widehat{\mu}\in \zra$, with the following relations, where $0\leq i,j\leq r-1$ and the indices of the $T_i$ and $T_j$ are taken modulo $r$:
\begin{enumerate}
    \item (Braid relations.) \textnormal{If }$r\geq 3$, $T_i T_{i+1} T_i=T_{i+1} T_i T_{i+1}$ and, when $|i-j|>1$, $T_i T_j=T_j T_i$.\label{braid}
    \item (Hecke relations.) $(T_i-k)(T_i+k^{-1})=0$.
    \item $\omega\omega^{-1}=1=\omega^{-1}\omega$ and $\omega T_i=T_{i+1}\omega$.
    \item For $\widehat{\mu},\widehat{\lambda}\in\zra$, $X^{\widehat{\mu}} X^{\widehat{\lambda}}=X^{\widehat{\lambda}} X^{\widehat{\mu}}$.
    \item $q_n$ is central. We will write $q=q_n^n$.
    \item $X^\delta=q_n^n=q$.
    \item (Cross relations.) For $\widehat{\mu}=\mu+a\delta\in\zra$,
    \begin{equation*}
        T_i X^{\widehat{\mu}} -X^{s_i *{\widehat{\mu}} }T_i=(k-k^{-1})\left( \dfrac{X^{\widehat{\mu}}-X^{s_i*{\widehat{\mu}}}}{1-X^{\alpha_i}}\right)\textnormal{ and }\omega X^{\widehat{\mu}}=q^{-\mu_r}X^{s_1\cdots s_{r-1}{\widehat{\mu}}}\omega.
    \end{equation*}
\end{enumerate}
\end{definition}

\begin{remark}
Precisely, the differences between $\mathbb{H}$ and the double affine Hecke algebra in \cite{ry} are: the use of the notation $k$ instead of $t$ (we have $t=k^2$), the enlargement of the center by the $n$th root $q_n$ of $q$, and the use of the field $\mathbb{K}^\enn$ instead of $\mathbb{C}$. 
\end{remark}

For $w\in W^{(1)}_{\textnormal{Cox}}$ with reduced expression $w=s_{i_1}\cdots s_{i_\ell}$, define
\begin{equation*}
    T_w=T_{i_1}\cdots T_{i_\ell}.
\end{equation*}
By a standard argument, since the $T_i$ satisfy the braid relations (\ref{braid}), $T_w$ is independent of the choice of reduced expression. 


The next definition is the variant of the double affine Hecke algebra used in \cite{ssv}.

\begin{definition}\label{daha meta}
$\mathbb{H}^{(n)}$ is the unital associative algebra over $\mathbb{K}^\enn$ generated by $q$, $\omega^\enn$, $(\omega^\enn)^{-1}$, $T_0, T_1, \dots, T_{r-1},$ and $X^{\widehat{\mu}}$ for ${\widehat{\mu}}\in P_a^\enn$, with the following relations, where $0\leq i,j\leq r-1$ and the indices of $T_i$, $T_{i+1}$, and $T_j$ are taken modulo $r$:
\begin{enumerate}
    \item (Braid relations.) $T_i T_{i+1} T_i=T_{i+1} T_i T_{i+1}$ and $T_i T_j=T_j T_i$ if $|i-j|>1$.
    \item (Hecke relations.) $(T_i-k)(T_i+k^{-1})=0$.
    \item $\omega^\enn(\omega^\enn)^{-1}=1=(\omega^\enn)^{-1}\omega^\enn$ and $\omega^\enn T_i=T_{i+1}\omega^\enn$.
    \item For $\widehat{\mu},\widehat{\lambda}\in P_a^\enn$, $X^{\widehat{\mu}} X^{\widehat{\lambda}}=X^{\widehat{\lambda}} X^{\widehat{\mu}}$.
    \item $q$ is central. 
    \item $X^{n^2\delta}=q^{n}.$\label{X and q}
    \item (Cross relations.) For ${\widehat{\mu}}=\mu+a\delta \in P_a^\enn$,
    \begin{equation*}
        T_i X^{\widehat{\mu}} -X^{s_i*{\widehat{\mu}} }T_i=(k-k^{-1})\left( \dfrac{X^{\widehat{\mu}}-X^{s_i^\enn *{\widehat{\mu}}}}{1-X^{\alpha_i^{(n)}}}\right)\textnormal{ and }\omega^\enn X^{\widehat{\mu}}=q^{-\mu_r}X^{s_1\cdots s_{r-1}{\widehat{\mu}}}\omega^\enn.
    \end{equation*}
\end{enumerate}
\end{definition}


The following observation is straightforward but important. 

\begin{proposition}\label{psi hecke}
We have an isomorphism of $\mathbb{K}^\enn$-algebras $\Psi^\enn: \mathbb{H}\rightarrow \mathbb{H}^{(n)}$ defined by
\begin{align*}
    \Psi^\enn(T_i)&=T_i,
    \\\Psi^\enn(\omega)&=\omega^\enn,
    \\\Psi^\enn(X^\mu)&=X^{n \mu}\textnormal{ for }\mu\in \mathbb{Z}^r,
    \\\Psi^\enn(X^\delta)&=X^{n^2\delta},
    \\\Psi^\enn(q_n)&=q.
\end{align*}

\end{proposition}

\begin{remark}
In \cite{ssv}, the element of $\mathbb{H}^\enn$ called $q$ (for the sake of this remark, let us call it $q_{SSV}$) is an $n$th root of the element we call $q$. Only $n$th powers of $q_{SSV}$ actually appear in \cite{ssv}, so we choose to use the current notation instead in order to minimize assumptions and make the examples in Appendix~\ref{appendix: nonsymmetric} appear more uniform. To summarize, using $q_{RY}$ to denote the $q$-parameter used in \cite{ry}, we have
\[q=q_{RY}=q_n^n=X^\delta\]
in $\mathbb{H}$, and
\[q=q_{SSV}^n\textnormal{ and }q^n=X^{n^2 \delta}\]
in $\mathbb{H}^\enn$. We also have
\[\Psi^\enn(q_n)=q\textnormal{ and }\Psi^\enn(q)=q^n.\]
\end{remark}

\begin{remark}
%
The conditions on $\Psi^\enn(X^\mu)$ and $\Psi^\enn(X^\delta)$ in Proposition~\ref{psi hecke} can be written more succinctly as
\begin{equation*}
    \Psi^\enn(X^{\mu+s\delta})=X^{\Psi^\enn(\mu+s\delta)}.
\end{equation*} 
\end{remark}

For $1\leq i\leq r-1$, we will use the notation $T_i^\vee=T_i\in \mathbb{H}$. Also define
\begin{equation*}
    T_0^\vee=(X^\theta T_{s_\theta})^{-1}\in\mathbb{H}.
\end{equation*}
We use the same notation for the corresponding element of $\mathbb{H}^\enn$, $T_0^\vee=\Psi^\enn(T_0^\vee)$. 

Define the following elements of $\mathbb{H}$, for $1\leq i\leq r$:
\begin{equation}\label{explicit Y}
    Y_i=Y^{\epsilon_i}=T_{i-1}^{-1}\cdots T_{1}^{-1}\omega T_{r-1}\cdots T_{i},
\end{equation}
where $T_{i-1}^{-1}\cdots T_{1}^{-1}$ and $T_{r-1}\cdots T_{i}$ are taken to have decreasing indices. These elements commute and are invertible. We then define, for $\widehat{\mu}=\sum \mu_i \epsilon_i+s\delta\in\zra$, 
\begin{equation*}
    Y^{\widehat{\mu}}=q^{-s}Y_1^{\mu_1}\cdots Y_r^{\mu_r}.
\end{equation*}
These elements satisfy the Bernstein-Zelevinsky cross relations
\begin{equation*}
    T_i^\vee Y^{\widehat{\mu}} -Y^{s_i*{\widehat{\mu}} }T_i^\vee=(k-k^{-1})\left( \dfrac{Y^{\widehat{\mu}}-Y^{s_i*{\widehat{\mu}}}}{1-Y^{-\alpha_i}}\right)
\end{equation*}
for $0\leq i\leq r-1$ and $\widehat{\mu}\in\zra$. (See \cite{is} for further discussion on the element $T_0^\vee$ and its analog in more general double affine Hecke algebras.) 

In $\mathbb{H}^{(n)},$ we have the corresponding elements
\begin{equation}\label{BZ Y}
    Y^{n\epsilon_i}=\Psi^\enn(Y^{\epsilon_i})=T_{i-1}^{-1}\cdots T_{1}^{-1}\omega^\enn T_{r-1}\cdots T_{i}
\end{equation}
and, for $\widehat{\mu}=\sum \mu_i\epsilon_i+s\delta\in P_a$, 
\begin{equation*}
    Y^{\Psi^\enn(\widehat{\mu})}=\Psi^\enn(Y^{\widehat{\mu}})
    =q^{-sn}(Y^{n\epsilon_1})^{\mu_1}\cdots (Y^{n\epsilon_r})^{\mu_r}.
\end{equation*}

For $0\leq i\leq r-1$ and $\widehat{\mu}\in P_a^\enn$, these elements satisfy
\begin{equation*}
    T_i^\vee Y^{\widehat{\mu}} -Y^{s_i^\enn *{\widehat{\mu}} }T_i^\vee=(k-k^{-1})\left( \dfrac{Y^{\widehat{\mu}}-Y^{s_i^\enn *{\widehat{\mu}}}}{1-Y^{-\alpha_i^\enn}}\right).
\end{equation*}

Let $\mathbb{F}^\enn[x^{\pm 1}]$ denote the space of Laurent polynomials in the variables $x_1, \dots, x_r$ over $\mathbb{F}^\enn=\mathbb{K}^\enn(q)$. For $\mu\in\mathbb{Z}^r$, we will use the notation
\begin{equation}\label{x notation}
    x^\mu=x_1^{\mu_1}\cdots x_r^{\mu_r}.
\end{equation}

For $j\in \mathbb{Z}$, recall that $r_n(j)\in\{0,\dots,n-1\}$ is the residue of $j$ modulo $n$. Define 
\begin{equation*}
    t_n(j)=j-r_n(j)\in n\mathbb{Z}.
\end{equation*}

The next theorem was proven in \cite{ssv}. The case $n=1$ is the well-known basic representation of \cite{cherednikBook}. 
\begin{theorem}
The following formulas define a representation $\pn$ of $\hn$ on $\mathbb{F}^\enn[x^{\pm 1}]$:
\begin{align}
    \pn(T_i) x^\lambda&= (k-k^{-1})\dfrac{1-x^{-t_n( (\lambda,\alpha_i)) \alpha_i}}{1-x^{n\alpha_i}}x^\lambda+G^\enn_{(\lambda,\alpha_i)}x^{s_i\lambda} \hspace{.25 in}(i\neq 0),\label{Ti action}
    \\\pn(T_0) x^\lambda &=(k-k^{-1})\dfrac{1-q^{-t_n(-(\lambda,\theta))}x^{t_n( -(\lambda,\theta)) \theta}}{1-q^{n}x^{-n\theta}}x^\lambda+G^\enn_{-(\lambda,\theta)}q^{(\lambda,\theta)}x^{s_\theta\lambda} ,\label{T0 action}
    \\\pn(X^\mu) x^\lambda &=x^{\mu+\lambda},
    \\\pn(\omega^\enn) x^\lambda &= q^{-\lambda_r}x^{s_1\cdots s_{r-1}\lambda},\label{omega rep}
    \\\pn(q)x^\lambda &= q x^\lambda
\end{align}
where $0\leq i\leq r-1$, $\lambda\in\mathbb{Z}^r$, and $\mu\in n\mathbb{Z}^r.$
\end{theorem}

\begin{remark}
Extend the notation (\ref{x notation}) by defining, for $\widehat{\mu}=\mu+s \delta\in \mathbb{Z}^r\oplus \mathbb{Z} n \delta$, 
\begin{equation*}
    x^{\widehat{\mu}}=q^{s/n} x_1^{\mu_1}\cdots x_r^{\mu_r}.
\end{equation*}
(Note the relationship with relation~\ref{X and q} of Definition~\ref{daha meta}.)Then equations (\ref{Ti action}) and (\ref{T0 action}) may be more compactly expressed as
\begin{equation*}
    \pn(T_i) x^\lambda = (k-k^{-1})\dfrac{1-x^{-t_n( (\lambda,\alpha_i)) \frac{1}{n}\alpha_i^\enn}}{1-x^{\alpha_i^\enn}}x^\lambda+G_{(\lambda,\alpha_i)}x^{s_i^\enn *\lambda} 
\end{equation*}
for $0\leq i\leq r-1$. Further, we may rewrite (\ref{omega rep}) as 
\begin{equation*}
    \pi^\enn(\omega^\enn)x^\lambda=x^{\omega^\enn * \lambda},
\end{equation*}
recalling the notation (\ref{omega def}). 
\end{remark}

Define $\sigma: \mathbb{Z}\rightarrow \mathbb{F}^\enn$ by
\begin{equation}\label{sigma def}
    \sigma(a)=\left\{
     \begin{array}{lr}
       k^{-1} & : a\in n\mathbb{Z}_{>0}\\
       G^\enn_{a} & : a\in\mathbb{Z}\setminus n\mathbb{Z}_{>0}
     \end{array}
   \right.
\end{equation}
for $a\in\mathbb{Z}$. (Although $\sigma$ depends on the positive integer $n$, we suppress it from the notation. We will not compare values of $\sigma$ for different values of $n$.) 
Note that $\sigma$ satisfies $\sigma(0)=G^\enn_0=k$ and 
\begin{equation}\label{sigma basic property}
    \sigma(a)\sigma(-a)=1
\end{equation}
for all $a\in\mathbb{Z}\setminus\{0\}$. 

From the theorem, we immediately have
\begin{corollary}\label{Ti on good monomials}
Let $\mu\in\mathbb{Z}^r$ and let $1\leq i\leq r-1$. If $0\leq (\mu,\alpha_i)\leq n$, then 
\begin{equation*}
    \pn(T_i) x^\mu=\sigma((\mu,\alpha_i))x^{s_i\mu}.
\end{equation*}
If $-n\leq (\mu,\alpha_i)\leq 0$, then 
\begin{equation*}
    \pn(T_i^{-1})x^\mu=\sigma((s_i \mu,\alpha_i))^{-1}x^{s_i\mu}.
\end{equation*} 
\end{corollary}


\subsection{The scalars \texorpdfstring{$\gamma(\widehat{\mu};\lambda)$}{gamma}}

For $\widehat{\mu}=\mu+sn^2\delta\in P_a^\enn$ and $\lambda\in\mathbb{Z}^r$, define the scalars
\begin{equation*}
    \gamma(\widehat{\mu};\lambda)=q^{-sn} \gamma(\mu;\lambda)=
    q^{-sn-(\mu,\lambda)/n}\prod_{\alpha\in\Phi_+}\left(\sigma((\lambda,\alpha))\right)^{(\mu,\alpha)/n},
\end{equation*}
where $\sigma$ is as in (\ref{sigma def}). 
We suppress the positive integer $n$ from this notation, but the context should be clear because we will have $\widehat{\mu}\in P_a^\enn$. 

Recalling the notation (\ref{angle bracket n}), we observe that
\begin{equation}\label{gamma clear q}
    \gamma(\widehat{\mu};\lambda)=q^{-\langle \widehat{\mu},\lambda \rangle/n}\prod_{\alpha\in\Phi_+}\sigma((\lambda,\alpha))^{(\mu,\alpha)/n}.
\end{equation}
We also note that
\begin{equation}\label{gamma inverse}
    \gamma(-\widehat{\mu};\lambda)=\gamma(\widehat{\mu};\lambda)^{-1}.
\end{equation}

Given $w\in W^\enn$ and a reduced expression $w=s_{i_1}^\enn\cdots s_{i_\ell}^\enn$, define the roots $\beta_1^\enn,\dots,\beta_\ell^\enn$ by
\begin{equation}\label{betas}
    \beta_j^\enn=s_{i_\ell}^\enn s_{i_{\ell-1}}^\enn\cdots s_{i_{j+1}}^\enn\alpha_{i_j}^\enn\in\widetilde{\Phi}_+^\enn.
\end{equation}
We now prove several technical but important lemmas involving $\gamma(\widehat{\mu};\lambda)$ and $\beta_j^\enn$. 

\begin{lemma}\label{gammas equal}
For $\widehat{\mu}\in P_a^\enn$, $\lambda\in\mathbb{Z}^r$, and $0\leq i\leq r-1$, whenever $s_i^\enn \lambda\neq \lambda$ we have
\begin{equation*}
    \gamma(s_i^\enn *\widehat{\mu};\lambda)=\gamma(\widehat{\mu};s_i^\enn \lambda).
\end{equation*}
Note that the action of $s_i^\enn$ on the left-hand side is the (linear) action on $P_a^\enn$, while the action of $s_i^\enn$ on the right-hand side is the (affine) action on $\mathbb{Z}^r$. This distinction is only relevant for $i=0$.
\end{lemma}
\begin{proof}
We will show that $\gamma(s_i^\enn * \widehat{\mu};s_i^\enn \lambda)=\gamma(\widehat{\mu};\lambda)$. Write $\widehat{\mu}=\mu+sn^2\delta$. 

First consider the case $i\neq 0$. We have
\begin{equation*}
    \gamma(s_i \widehat{\mu};s_i \lambda)
    = q^{-\langle s_i\widehat{\mu},s_i\lambda \rangle/n}\prod_{\alpha\in\Phi_+}\sigma((s_i\lambda,\alpha))^{(s_i\mu,\alpha)/n}
    =q^{-\langle \widehat{\mu},\lambda \rangle/n}\prod_{\alpha\in\Phi_+}\sigma((\lambda,s_i \alpha))^{(\mu,s_i\alpha)/n}.
\end{equation*}
Since $s_i\alpha_i=-\alpha_i$ and $s_i$ permutes the other elements of $\Phi_+$, we see that 
\begin{equation}\label{si gamma quotient}
    \dfrac{\gamma(s_i \widehat{\mu};s_i \lambda)}{\gamma( \widehat{\mu}; \lambda)}=\dfrac{\sigma((\lambda,-\alpha_i))^{(\mu,-\alpha_i)}}{\sigma((\lambda,\alpha_i))^{(\mu,\alpha_i)}}=(\sigma(\lambda_{i+1}-\lambda_i)\sigma(\lambda_i-\lambda_{i+1}))^{(\mu,-\alpha_i)}.
\end{equation}
Since $s_i\lambda\neq \lambda$, we have $\lambda_i-\lambda_{i+1}\neq 0$, so by (\ref{sigma basic property}), (\ref{si gamma quotient}) is equal to $1$. 

Now consider the case $i=0$. Note that
\begin{equation*}
    s_0^\enn *\widehat{\mu}=s_\theta \mu+(sn^2+(\mu_1-\mu_r)n)\delta.
\end{equation*}
We have
\begin{align}
    \gamma(s_0^\enn * \widehat{\mu};s_0^\enn \lambda)
    &=q^{-\langle s_0^\enn\widehat{\mu},s_0^\enn\lambda \rangle/n}
    \prod_{\alpha\in\Phi_+}\sigma((s_0^\enn\lambda,\alpha))^{(s_\theta \mu,\alpha)/n}
    \\&=q^{-\langle \widehat{\mu}, \lambda \rangle/n}\prod_{\alpha\in\Phi_+}\sigma((s_0^\enn\lambda,\alpha))^{(s_\theta \mu,\alpha)/n} \label{s0 gamma}
\end{align}
Write
\begin{equation}\label{unmarked gamma}
    \gamma( \widehat{\mu};\lambda)=q^{-\langle \widehat{\mu}, \lambda \rangle/n}
    \prod_{\beta\in\Phi_+}\sigma((\lambda,\beta))^{(\mu,\beta)/n}.
\end{equation}
We claim that for all $\alpha\in\Phi_+$, the term in (\ref{s0 gamma}) corresponding to $\alpha$ is equal to the term in (\ref{unmarked gamma}) corresponding to $\beta=\pm s_\theta \alpha$ (whichever is in $\Phi_+$). We leave this for the reader to check in most cases: we will prove it only for the most complicated case, $\alpha=\theta$. We have
\begin{equation*}
    \sigma(s_0^\enn \lambda,\theta))^{(s_\theta\mu,\theta)/n}
    =\sigma((\lambda_r+n)-(\lambda_1-n))^{(\mu,s_\theta \theta)/n}
    =\sigma(2n-(\lambda_1-\lambda_r))^{-(\mu,\theta)/n}.
\end{equation*}
Then we must only show, letting $a=\lambda_1-\lambda_r$, that 
\begin{equation*}
    \sigma(2n-a)\sigma(a)=1.
\end{equation*}
If $a$ is not a multiple of $n$, we get 
\begin{equation*}
    \sigma(2n-a)\sigma(a)= G_{2n-a} G_a =1.
\end{equation*}
If $a$ is any multiple of $n$ other than $n$, then one of $\{a,2n-a\}$ is positive and the other is not, so 
\begin{equation*}
    \sigma(2n-a)\sigma(a)=kk^{-1}=1.
\end{equation*}
Finally, we observe that because $s_0^\enn \lambda\neq n$, we must have
\begin{equation*}
    a=\lambda_1-\lambda_r\neq n,
\end{equation*}
so the result holds. 
\end{proof}

\begin{lemma}\label{gamma not zero}
\begin{enumerate}
    \item Let $\lambda\in A^\enn$ and $\widehat{\alpha}=n(\epsilon_i-\epsilon_j)+sn^2\delta\in\widetilde{\Phi}^\enn$, where $i<j$. The power of $q$ in $\gamma(\widehat{\alpha},\lambda)$ is zero if and only if $s=0$ and $\lambda_i=\lambda_j$, or $s=-1$ and $\lambda_i-\lambda_j=n$. 
    \item For all $\widehat{\alpha}\in\widetilde{\Phi}^\enn$ and $\lambda\in \mathbb{Z}^r$, $\gamma(\widehat{\alpha}; \lambda)\neq 1$. 
\end{enumerate}
\end{lemma}
\begin{proof}
Part 1 follows from (\ref{gamma clear q}) and the definition of $A^\enn$. For Part 2, note that by Lemma~\ref{gammas equal}, it suffices to consider the case where $\lambda\in A^\enn$. By (\ref{gamma inverse}), we may restrict to $\widehat{\alpha}$ as in Part 1 of this lemma. Further, by Part 1 of this lemma, we only need to consider two cases for $\widehat{\alpha}$: either $s=0$ and $\lambda_i=\lambda_j$, or $s=-1$ and $\lambda_i-\lambda_j=n$. We will sketch the proof in the first case, leaving the second (similar) case to the reader. 

Suppose that we are in the case where $s=0$ and $\lambda_i=\lambda_j$. Let $v=\lambda_i=\lambda_j$ be the common value. Write
\begin{equation*}
    \gamma(\widehat{\alpha};\lambda)=
    \prod_{\beta\in\Phi_+}\sigma((\lambda,\beta))^{(\alpha,\beta)/n}.
\end{equation*}
Only the terms with $(\alpha,\beta)\neq 0$ are relevant, so we must only consider $\beta=\epsilon_c-\epsilon_d$ where at least one of $c$ or $d$ is in $\{i,j\}$. 
For all $b<i$ (respectively $b>j$), the terms for $\beta=\epsilon_b-\epsilon_i$ (resp. $\beta=\epsilon_i-\epsilon_b$) and $\beta=\epsilon_b-\epsilon_j$ (resp. $\beta=\epsilon_j-\epsilon_b$) will cancel. Since $\lambda\in A^\enn$, if $i\leq b\leq j$, then $\lambda_b=v$. 
Then we have
\begin{equation*}
        \gamma(\widehat{\alpha};\lambda)=
        \sigma(0)^2\prod_{i<b<j}\sigma(0)^{(\epsilon_i-\epsilon_j,\epsilon_i-\epsilon_b)}\sigma(0)^{(\epsilon_i-\epsilon_j,\epsilon_b-\epsilon_j)}
        =\sigma(0)^{2(j-i)}=k^{2(j-i)}\neq 1.
\end{equation*}
\end{proof}

\begin{lemma}\label{q power}
Let $\lambda\in A^\enn$ and $\widehat{\alpha}\in \widetilde{\Phi}_+^\enn$.
\begin{enumerate}
    \item The power of $q$ in $\gamma({-\widehat{\alpha}};\lambda)$ is nonnegative. If $\langle \widehat{\alpha},\lambda\rangle\neq 0$, then the power of $q$ in $\gamma({-\widehat{\alpha}};\lambda)$ is strictly positive. 
    \item Let $w\in W_{\textnormal{Cox}}^\enn$, and let $w=s_{i_1}\cdots s_{i_\ell}$ be a reduced expression for $w$. Let $\beta_j^\enn$ be as defined in \eqref{betas}. For $1\leq j\leq \ell$, $\langle \beta_j^\enn,\lambda\rangle \neq 0$. In particular, the power of $q$ in $\gamma(-\beta_j^\enn;\lambda)$ is strictly positive. 
\end{enumerate}
\end{lemma}
\begin{proof}
The first part follows from (\ref{gamma clear q}) and Lemma~\ref{gamma not zero}. For the second part, note that because $w^\enn$ is the shortest element of $W^\enn_{\textnormal{Cox}}$ with $w^\enn \lambda=\mu$, we must have
\begin{equation}\label{si matters}
    s_{i_j}^\enn \cdots s_{i_\ell}^\enn \lambda\neq s_{i_{j+1}}^\enn \cdots s_{i_\ell}^\enn \lambda
\end{equation}
for all $1\leq j\leq \ell$. By (\ref{general reflection}) and (\ref{si as reflection}), (\ref{si matters}) implies
\begin{equation*}
    0\neq \langle \alpha_{i_j}^\enn , s_{i_{j+1}}^\enn \cdots s_{i_\ell}^\enn \lambda \rangle =\langle  s_{i_\ell}^\enn \cdots s_{i_{j+1}}^\enn \alpha_{i_j}^\enn  , \lambda  \rangle =\langle \beta_j^\enn,\lambda\rangle.
\end{equation*}
Then the result follows from the first part of this lemma and the fact that $\beta_1^\enn, \dots, \beta_\ell^\enn$ are positive roots. 
\end{proof}

\section{(Nonsymmetric) SSV polynomials}\label{section: ssv polys}

The following theorem was stated in \cite{ssv} and will be proven in \cite{ssv2}.

\begin{theorem}\label{characterization}
There exists a unique family of elements $\{E_\mu^\enn\in \mathbb{F}^\enn[x^{\pm 1}]: \mu\in\mathbb{Z}^r\}$
such that for each $\mu\in\mathbb{Z}^r$, the coefficient of $x^\mu$ in $E_\mu^\enn$ is 1 and, for all $\widehat{\nu}\in P_a^\enn$,
\begin{equation}\label{Y eigenvalues}
    \pn(Y^{\widehat{\nu}})E^\enn_\mu=\gamma(\widehat{\nu};\mu)E_\mu^\enn.
\end{equation}
This last condition is equivalent to the condition that, for $1\leq i\leq r$ and $\mu\in\mathbb{Z}^r$,
\begin{equation}\label{Y_i evals}
    \pn(Y^{n\epsilon_i}) E_\mu=q^{-\mu_i}\left(\prod_{1\leq j<i}\sigma((\mu,\epsilon_j-\epsilon_i))\right)^{-1}\left(\prod_{i< j\leq r-1}\sigma((\mu,\epsilon_i-\epsilon_j))\right)E_\mu.
\end{equation}
\end{theorem}

Call the family $\{E_\mu^\enn:\mu\in\mathbb{Z}^r\}$ the (nonsymmetric) \emph{SSV polynomials}. We will sometimes write 
\begin{equation*}
    E_\mu^\enn=E_\mu^\enn(x;q,k)
\end{equation*}
to emphasize the dependence on the parameters $q$ and $k$. For $n\nu\in n\mathbb{Z}^r$, these satisfy
\begin{equation*}
    E_{n\nu}^\enn(x_1,\dots,x_r;q,k)=E_\nu(x_1^n, \dots, x_r^n; q^{n},k),
\end{equation*}
where $E_\nu(x_1,\dots,x_r;q,k)$ is the standard nonsymmetric Macdonald polynomial corresponding to $GL_n$. For examples, see Appendix~\ref{appendix: nonsymmetric}.


\begin{proposition}\label{monomial}
Let $\lambda\in A^\enn$. Then $E_\lambda^\enn=x^\lambda$.
\end{proposition}
\begin{proof}
We will calculate this explicitly, using Corollary~\ref{Ti on good monomials} and the explicit form (\ref{explicit Y}) of the elements $Y_i$. Let $1\leq i\leq r-1$. 
Note that for all $j$ with $i\leq j\leq r-1$, we have
\begin{equation*}
    (s_{j-1}\cdots s_i\lambda,\alpha_i)=\lambda_i-\lambda_{j+1}.
\end{equation*}
Since $\lambda\in A^\enn$,
\begin{equation*}
    0\leq (s_{j-1}\cdots s_i\lambda,\alpha_i)\leq n.
\end{equation*}
Then by the first statement of Corollary~\ref{Ti on good monomials},
\begin{align*}
    \pn(T_{r-1})\cdots \pn(T_i) x^\lambda &=\sigma((s_{r-2}\cdots s_i\lambda,\alpha_{r-1}))\cdots \sigma((\lambda,\alpha_i)) x^{s_{r-1}\cdots s_i\lambda}
    \\&= \sigma((\lambda,\epsilon_i-\epsilon_r))\cdots\sigma((\lambda,\epsilon_i-\epsilon_{i+1})) x^{s_{r-1}\cdots s_i\lambda}.
\end{align*}
Next, we compute
\begin{equation*}
    \omega^\enn x^{s_{r-1}\cdots s_i\lambda}=q^{-(\epsilon_r,s_{r-1}\cdots s_i\lambda)} x^{s_1\cdots s_{r-1}s_{r-1}\cdots s_i\lambda}=q^{-\lambda_i}x^{s_1\cdots s_{i-1}\lambda}.
\end{equation*}
Applying the second statement of the lemma, we have
\begin{align*}
    \pn(T_{i-1}^{-1})\cdots \pn(T_{1}^{-1})x^{s_1\cdots s_{i-1}\lambda} &= \sigma((\lambda,\alpha_{i-1}))^{-1}\cdots\sigma((s_2\cdots s_{i-1}\lambda,\alpha_1))^{-1}x^{\lambda}
    \\ &=\sigma((\lambda,\alpha_{i-1}))^{-1}\cdots\sigma((\lambda,s_{i-1}\cdots s_2\alpha_1))^{-1}x^{\lambda}
    \\ &=\sigma((\lambda,\epsilon_{i-1}-\epsilon_i))^{-1}\cdots\sigma((\lambda,\epsilon_1-\epsilon_i))^{-1}x^{\lambda}.
\end{align*}
Then 
\begin{equation*}
    \pn(Y^{n\epsilon_i}) x^\lambda=q^{-\lambda_i}\left(\prod_{1\leq j<i}\sigma((\lambda,\epsilon_j-\epsilon_i))\right)^{-1}\left(\prod_{i< j\leq r-1}\sigma((\lambda,\epsilon_i-\epsilon_j))\right)x^\lambda=\gamma(n\epsilon_i;\lambda)x^\lambda,
\end{equation*}
by (\ref{Y_i evals}). 
By Theorem~\ref{characterization}, $E_\lambda^\enn=x^\lambda$.
\end{proof}

For $0\leq i\leq r-1$, we define the \emph{polynomial intertwiner} 
$S_i\in \mathbb{H}$ by
\begin{equation*}
    S_i=T_i^\vee(1-Y^{-\alpha_i}) +(k^{-1}-k).
\end{equation*}
We use the same notation for the element of $\mathbb{H}^\enn$ defined by
\begin{equation*}
    S_i=\Psi^\enn(S_i)=T_i^\vee(1-Y^{-\alpha_i^\enn}) +(k^{-1}-k).
\end{equation*}
The following is the essential property of the polynomial form intertwiners, where we use $f\sim g$ to mean that $f$ is a nonzero multiple of $g$. 
\begin{proposition}\label{intertwiner property S}
We have
\begin{equation}\label{intertwiners and y S}
    S_i Y^{\widehat{\nu}}=Y^{s_i *\widehat{\nu}}S_i
\end{equation}
in $\mathbb{H}$, where $0\leq i\leq r-1$ and $\widehat{\nu}\in\zra$. In particular, applying $\Psi^\enn$ and changing indices, we have
\begin{equation}\label{intertwiners and y 2 S}
    S_i Y^{\widehat{\nu}}=Y^{s_i^\enn *\widehat{\nu}}S_i
\end{equation}
in $\mathbb{H}^\enn$ as well, where now $0\leq i\leq r-1$ and $\widehat{\nu}\in P_a^\enn$.
Further, for any $\mu\in\mathbb{Z}^r$ and $0\leq i\leq r-1$ such that $s_i^\enn \mu\neq \mu$,
\begin{equation}\label{intertwiner on E S}
    \pn(S_i) E_\mu^\enn\sim E_{s_i^\enn \mu}^\enn.
\end{equation}
\end{proposition}
\begin{proof}
The proof of (\ref{intertwiners and y S}) is a straightforward exercise using the relations (\ref{BZ Y}), and (\ref{intertwiners and y 2 S}) follows from (\ref{intertwiners and y S}). For (\ref{intertwiner on E S}), we simply compute that for ${\widehat{\nu}}\in P_a^\enn$,
\begin{align*}
    \pn(Y^{{{\widehat{\nu}}}}) \pn(S_i) E_\mu^\enn
    &=\pn(S_i) \pn(Y^{s_i^\enn * {{{\widehat{\nu}}}}}) E_\mu^\enn 
    \\&=\gamma(s_i^\enn * \widehat{\nu};\mu)\pn(S_i)E_\mu^\enn
    \\&=\gamma(\widehat{\nu};s_i^\enn \mu) \pn(S_i)E_\mu^\enn,
\end{align*}
where the first equality follows from (\ref{intertwiners and y 2 S}), the second from (\ref{Y eigenvalues}), and the third from Lemma~\ref{gammas equal}. Then Theorem~\ref{characterization} implies that $\pn(S_i) E_\mu^\enn$ is a multiple of $E_{s_i^\enn \mu}^\enn$. To see that it is nonzero, we will show that $\pn(S_i)^2 E_\mu^\enn\neq 0$. 
By a direct computation using (\ref{BZ Y}), we see that
\begin{equation*}
    S_i^2=(k^2+k^{-2})-(Y^{\alpha_i}+Y^{-\alpha_i}).
\end{equation*}
Then
\begin{equation*}
    \pn(S_i)^2 E_\mu^\enn=((k^2+k^{-2})-(\gamma({\alpha_i^\enn};\mu)+\gamma({\alpha_i^\enn};\mu)^{-1}))E_\mu^\enn.
\end{equation*}
In order for this scalar to be zero, we would need $\gamma({\alpha_i^\enn};\mu)=k^{\pm 2}$. This is impossible: since $s_i^\enn \mu\neq \mu$, we have $\langle \alpha_i^\enn,\mu\rangle\neq 0$, and therefore (\ref{gamma clear q}) makes it clear that $\gamma({\alpha_i^\enn};\mu)$ has a nonzero power of $q$. 
\end{proof}

In order to use the same intertwiners as in \cite{ry}, we must first define localized versions of the relevant Hecke algebras. Let $\mathbb{H}_Y$ be the subalgebra of $\mathbb{H}$ generated by $q_n, T_0^\vee, \dots, T_{r-1}^\vee$, and $Y^{\widehat{\mu}}$ for $\widehat{\mu}\in P_a$. (Note that this is not all of $\mathbb{H}$. $\mathbb{H}_Y$ does include the element $\omega$ and $X^{\mu}$ for $\mu\in Q$, but does not include all $X^{\mu}$ for $\mu\in \mathbb{Z}^r$.) 
Consider the multiplicatively closed subset $S$ of $\mathbb{H}_Y$ generated by
\begin{equation*}
    \{1-Y^{\widehat{\alpha}}: \widehat{\alpha}\in\widetilde{\Phi}\}.
\end{equation*}
$S$ satisfies the right Ore condition: this follows from the commutativity of the $Y$'s and the identity
\begin{equation*}
    T_i^\vee (1-Y^{-\alpha_i})(1-Y^{\widehat{\alpha}})(1-Y^{s_i\widehat{\alpha}})=(1-Y^{\widehat{\alpha}})\left(T_i^\vee (1-Y^{-\alpha_i})(1-Y^{\widehat{\alpha}})-(k-k^{-1})(Y^{s_i\widehat{\alpha}}-Y^{\widehat{\alpha}})   \right)
\end{equation*}
for $0\leq i\leq r-1$ and $\widehat{\alpha}\in\widetilde{\Phi}$, which can be checked directly using (\ref{BZ Y}). This allows us to define $\mathbb{H}_{Y,loc}$ as the localization of $\mathbb{H}_Y$ at $S$. 

We may similarly define $\mathbb{H}_Y^\enn$ and the localization $\mathbb{H}_{Y,loc}^\enn$ of $\mathbb{H}_Y^\enn$ at the multiplicatively closed subset generated by
\begin{equation*}
    \{1-Y^{\widehat{\alpha}}: \widehat{\alpha}\in\widetilde{\Phi}^\enn\}.
\end{equation*}
We have an isomorphism
\begin{equation*}
    \Psi^\enn:\mathbb{H}_{Y,loc}\rightarrow \hn_{Y,loc},
\end{equation*}
defined in the obvious way. 

For $0\leq i\leq r-1$, we may now define the \emph{intertwiners} $\uptau_i^\vee\in \mathbb{H}_{Y,loc}$ by
\begin{equation}\label{intertwiner def}
    \uptau_i^\vee=T_i^\vee +\dfrac{k^{-1}-k}{1-Y^{-\alpha_i}}=S_i(1-Y^{-\alpha_i})^{-1}.
\end{equation}
These are precisely the same intertwiners used in \cite{ry}. We also have the version for $\mathbb{H}_{Y,loc}^\enn$, for which we use the same notation:
\begin{equation*}
    \uptau_i^\vee=\Psi^\enn(\uptau_i^\vee)=T_i^\vee +\dfrac{k^{-1}-k}{1-Y^{-\alpha_i^\enn}}=S_i(1-Y^{-\alpha_i^\enn})^{-1}.
\end{equation*}

We now wish to construct a version of the representation $\pi^\enn$ for $H_{Y,loc}^\enn$. First, define
\begin{equation*}
    \mathcal{E}=\textnormal{span}_{K^\enn(q)}\{E_\mu^\enn: \mu\in\mathbb{Z}^r\}.
\end{equation*}
We will soon see that $\mathcal{E}=\mathbb{F}^\enn[x^{\pm 1}]$. The proof of this result can be found in \cite{ssv2}, but it is easy to see from Theorem~\ref{main}, so we include it here as well.

Note that $q$, $S_i$ for $0\leq i\leq r-1$, and $Y^{\widehat{\mu}}$ for $\widehat{\mu}\in P_a^\enn$ generate $\mathbb{H}_Y^\enn$ and, by Proposition~\ref{intertwiner property S} and Theorem~\ref{characterization}, they preserve the subspace $\mathcal{E}$. Define a representation of $\mathbb{H}_{Y,loc}^\enn$ on $\mathcal{E}$, which we call $\pi^\enn$ by abuse of notation, by 
\begin{equation*}
    \pn(z)=\pn(z)|_{\mathcal{E}}
\end{equation*}
for $z\in\mathbb{H}_Y$ and 
\begin{equation*}
    \pn((1-Y^{\widehat{\alpha}})^{-1})=(\pn(1-Y^{\widehat{\alpha}})|_{\mathcal{E}})^{-1}
\end{equation*}
for $\widehat{\alpha}\in\widetilde{\Phi}$. 

\begin{proposition}
The representation $\pn$ of $\mathbb{H}_{Y,loc}^\enn$ on $\mathcal{E}$ is well-defined. 
\end{proposition}
\begin{proof}
We must only observe that $\pn(1-Y^{\widehat{\alpha}})$ is invertible on $\mathcal{E}$. By Theorem~\ref{characterization}, $\pn(1-Y^{\widehat{\alpha}})$ is a diagonalizable linear operator on $\mathcal{E}$, and by Lemma~\ref{gamma not zero}, its eigenvalues are all nonzero.
\end{proof}

Then $\pi^\enn(\uptau_i^\vee)$ is well-defined. We will work with $\uptau_i^\vee$ instead of $S_i$ from now on, in order to match \cite{ry}. 

We immediately have the following version of Proposition~\ref{intertwiner property S}. 

\begin{proposition}\label{intertwiner property}
We have
\begin{equation*}
    \uptau_i^\vee Y^{\widehat{\nu}}=Y^{s_i *\widehat{\nu}}\uptau_i^\vee
\end{equation*}
in $\mathbb{H}_{Y,loc}$, where $0\leq i\leq r-1$ and $\widehat{\nu}\in\zra$. Further,
\begin{equation*}
    \uptau_i^\vee Y^{\widehat{\nu}}=Y^{s_i^\enn *\widehat{\nu}}\uptau_i^\vee
\end{equation*}
in $\mathbb{H}_{Y,loc}^\enn$, where now $0\leq i\leq r-1$ and $\widehat{\nu}\in P_a^\enn$.
For any $\mu\in\mathbb{Z}^r$ and $0\leq i\leq r-1$ such that $s_i^\enn \mu\neq \mu$,
\begin{equation*}
    \pn(\uptau_i^\vee) E_\mu^\enn\sim E_{s_i^\enn \mu}^\enn.
\end{equation*}
\end{proposition}

For $w\in W^{(1)}_{\textnormal{Cox}}$ with reduced expression
\begin{equation*}
    w=s_{i_1}\cdots s_{i_m},
\end{equation*}
define
\begin{equation}\label{general intertwiners}
    \uptau_w^\vee=\uptau_{i_1}^\vee\cdots \uptau_{i_m}^\vee.
\end{equation}
It is a straightforward but tedious calculation to check that the operators $\uptau_i$ satisfy the braid relations of the $T_i$, and therefore $\uptau_w^\vee$ is independent of the choice of reduced expression of $w$. We will not actually need to use this fact, but the notation $\uptau_w^\vee$ is convenient. 

We have the following corollary of (\ref{general intertwiners}) and Proposition~\ref{intertwiner property}.

\begin{corollary}\label{general int application}
For $w\in W^\enn_{\textnormal{Cox}}$ and $\mu\in\mathbb{Z}^r$,
\begin{equation*}
    \pn(\uptau_w^\vee) E_\mu^\enn\sim E_{w\mu}^\enn.
\end{equation*}
\end{corollary}

The next proposition follows from Lemma~\ref{lambda and w} and Corollary~\ref{general int application}. 

\begin{proposition}\label{fundamental domain}
Let $\mu\in\mathbb{Z}^r$. There exist $\lambda\in A^\enn$ and $w\in W^\enn_{\textnormal{Cox}}$ such that
\begin{equation*}
    \pn(\uptau_w^\vee) x^\lambda=
    \pn(\uptau_w^\vee) E_\lambda \sim E_\mu.
\end{equation*}
\end{proposition}


For any alcove walk $p$, recall the notation
\begin{equation*}
    \textnormal{end}(p)=\tau(\textnormal{wt}(p))\phi(p)\in W_{\textnormal{Cox}}
\end{equation*}
from (\ref{end, weight, phi}), 
where $\textnormal{wt}(p)\in Q$ and $\phi(p)\in W_0$. 

We have the following result from \cite{ry}:

\begin{theorem}\label{ry main}
Let $w\in W^{(1)}_{\textnormal{Cox}}$ and $u\in W_0$. Write $w=s_{i_1}\cdots s_{i_\ell}$, a reduced expression. For $1\leq j\leq \ell$, let $\beta_j=\beta_j^{(1)}$ as defined in \eqref{betas}.
Then in $\mathbb{H}_{Y,loc}$, we have
\begin{equation*}
    T_u\uptau_w^\vee =\sum_{p\in\mathcal{B}(u,\vec{w})}X^{\textnormal{wt}(p)}T_{\phi(p)}\left (\prod_{j\in f^+(p)}\dfrac{k^{-1}-k}{1-Y^{-\beta_j}} \right) \left(\prod_{j\in f^-(p)}\dfrac{(k^{-1}-k)Y^{-\beta_j}}{1-Y^{-\beta_j}} \right).
\end{equation*}
\end{theorem}



\begin{remark}
The notation above differs from that in \cite{ry}. For $w\in W_{\textnormal{Cox}}$ with expression $w=s_{i_1}\cdots s_{i_\ell}$ (not necessarily reduced), Ram and Yip used the notation
\begin{equation}\label{x to the w}
    X^w=(T_{i_1}^\vee)^{\epsilon_1}\cdots (T_{i_\ell}^\vee)^{\epsilon_\ell},
\end{equation}
where $\epsilon_j$ is $1$ if the $j$th step of the corresponding alcove walk is a positive crossing and $-1$ if it is a negative crossing. This is independent of the choice of expression for $w$: see \cite{gortz}. 

The dictionary between the two choices of notation is
\begin{equation*}
    X^{\tau(\lambda)u}=X^\lambda T_u.
\end{equation*}
One can prove this by using the fact that (\ref{x to the w}) is independent of the expression for $w$ and the following standard characterization of $X^\lambda$ for $\lambda\in Q$. If $\lambda\in Q$ is antidominant and $\tau(\lambda)=s_{i_1}\cdots s_{i_\ell}$ is a reduced expression, then
\begin{equation*}
    X^\lambda=T_{i_1}^\vee \cdots T_{i_\ell}^\vee.
\end{equation*}
For general $\lambda\in Q$, we write $\lambda=\lambda_+-\lambda_-$ for antidominant $\lambda_+,\lambda_-\in Q$, and observe that
\begin{equation*}
    X^\lambda=X^{\lambda_+}(X^{\lambda_-})^{-1}.
\end{equation*}
The result then follows by expressing $\tau(\lambda) u$ in terms of an alcove walk that begins at $1$, proceeds to $\tau(\lambda_+)$, then to $\tau(\lambda)=\tau(\lambda_+)\tau(-\lambda_-)$, and finally to $\tau(\lambda)u$ (where each step follows a shortest possible route). The crossings from $1$ to $\tau(\lambda_+)$ will all be positive, the crossings from $\tau(\lambda_+)$ to $\tau(\lambda)$ will all be negative, and the crossings from $\tau(\lambda)$ to $\tau(\lambda)u$ will all be positive, giving
\begin{equation*}
    X^{\tau(\lambda)u}=X^{\lambda_+} (X^{\lambda_-})^{-1} T_u= X^\lambda T_u.
\end{equation*}
\end{remark}

Theorem~\ref{ry main}, in conjunction with the representation $\pn$, gives:

\begin{theorem}\label{main}
Let $\mu\in\mathbb{Z}^r$, and choose the unique $\lambda\in A^\enn$ and (the unique) shortest $w^\enn\in W^\enn_{\textnormal{Cox}}$ such that $w^\enn\lambda=\mu$. (This is possible by Proposition~\ref{fundamental domain}.) 
Write $w^\enn=s_{i_1}^\enn\cdots s_{i_\ell}^\enn$, a reduced expression. 
Then define
\begin{equation*}
    w=(\Psi^\enn)^{-1}(w^\enn)=s_{i_1}\cdots s_{i_\ell},
\end{equation*}
the right-hand side being a reduced expression for $w$. 
For $1\leq j\leq \ell$, let $\beta_j^\enn$ be defined as in (\ref{betas}): 
\begin{equation*}
    \beta_j^\enn=s_{i_\ell}^\enn s_{i_{\ell-1}}^\enn\cdots s_{i_{j+1}}^\enn\alpha_{i_j}^\enn.
\end{equation*}
For each $p\in \mathcal{B}(\vec{w})$, write a reduced expression for $\phi(p)\in W_0$ as follows (where $t=\ell(\phi(p))$ depends on $p$):
\begin{equation*}
    \phi(p)=s_{u_{p,1}}\cdots s_{u_{p,t}}.
\end{equation*}
Then $E_\mu^\enn$ is a nonzero scalar multiple of
\begin{align}\label{result}
    \sum_{p\in \mathcal{B}(\vec{w})} & x^{n\textnormal{wt}(p)+\phi(p)\lambda}\left (\prod_{a=1}^{\ell(\phi(p))}\sigma((\lambda,s_{u_{p,t}}\cdots s_{u_{p,a+1}} \alpha_{u_{p,a}})) \right)\cdot \notag
    \\& \cdot\left (\prod_{j\in f^+(p)}\dfrac{k^{-1}-k}{1-\gamma({-\beta_j^\enn};\lambda)} \right) \left(\prod_{j\in f^-(p)}\dfrac{(k^{-1}-k)\gamma({-\beta_j^\enn};\lambda)}{1-\gamma({-\beta_j^\enn};\lambda)} \right).
\end{align}
More precisely, if $\widetilde{p}$ is the unique unfolded walk in $\mathcal{B}(\vec{w})$, then (\ref{result}) is equal to
\begin{equation}\label{scalar}
    \left (\prod_{a=1}^{\ell(\phi(\widetilde{p}))}\sigma((\lambda,s_{u_{\widetilde{p},t}}\cdots s_{u_{\widetilde{p},a+1}} \alpha_{u_{\widetilde{p},a}})) \right)E_\mu^\enn.
\end{equation}
\end{theorem}
\begin{proof}
By Proposition~\ref{fundamental domain}, we have $E_\mu^\enn\sim \pn(\uptau_w^\vee) x^\lambda$. Applying Theorem~\ref{ry main}, we have
\begin{equation*}
    \sum_{p\in \mathcal{B}(\vec{w})}x^{n\textnormal{wt}(p)}\pn(T_{\phi(p)}) x^\lambda \left (\prod_{j\in f^+(p)}\dfrac{k^{-1}-k}{1-\gamma({-\beta_j^\enn};\lambda)} \right) \left(\prod_{j\in f^-(p)}\dfrac{(k^{-1}-k)\gamma({-\beta_j^\enn};\lambda)}{1-\gamma({-\beta_j^\enn};\lambda)} \right)\sim E_\mu^\enn.
\end{equation*}
Fix $p\in \mathcal{B}(\vec{w})$. We must only compute (dropping the dependence on $p$ in the notation on the right-hand side)
\begin{equation*}
    \pn(T_{\phi(p)}) x^\lambda=\pn(T_{s_{u_1}})\cdots \pn(T_{s_{u_t}})x^\lambda.
\end{equation*}
This is similar to the proof of Proposition~\ref{monomial}, although slightly less explicit. Since $\phi(p)=s_{u_1}\cdots s_{u_t}$ is a reduced expression, for $1\leq a\leq t$ we have
\begin{equation*}
    s_{u_t}\cdots s_{u_{a+1}}\alpha_{u_a} >0
\end{equation*}
(see, for instance, Section 1.7 of \cite{humphCox}). 
Since $\lambda$ is dominant, this implies that 
\begin{equation*}
    0\leq (\lambda,s_{u_t}\cdots s_{u_{a+1}}\alpha_{u_a} )\leq n,
\end{equation*}
or equivalently
\begin{equation*}
    0\leq (s_{u_{a+1}}\cdots s_{u_t}\lambda,\alpha_{u_a} )\leq n.
\end{equation*}
Then by Corollary~\ref{Ti on good monomials}, 
\begin{align*}
    \pn(T_{s_{u_1}})\cdots \pn(T_{s_{u_t}})x^\lambda
    &=\sigma((s_{u_2}\cdots s_{u_t}\lambda,\alpha_{u_1}))\cdots \sigma((\lambda,\alpha_{u_t}))x^{s_{u_1}\cdots s_{u_t}\lambda}
    \\&=\prod_{a=1}^{\ell(\phi(p))}\sigma((\lambda,s_{u_{t}}\cdots s_{u_{a+1}} \alpha_{u_{a}}))x^{\phi(p)(\lambda)}
\end{align*}
matching (\ref{result}). 

To obtain (\ref{scalar}), we observe that because $w^\enn$ is an element of $W^\enn_{\textnormal{Cox}}$ of shortest length such that $w^\enn\lambda=\mu$, and any alcove walk other than $\widetilde{p}$ is nontrivially folded and therefore corresponds to a proper subword of $w=(\Psi^\enn)^{-1}(w^\enn)$, $\widetilde{p}$ is the unique alcove walk $p\in\mathcal{B}(\vec{w})$ with
\[
    \Psi^\enn(\textnormal{end}(p))(\lambda)=\mu.
\]
\end{proof}

Let $\mu$, $\lambda$, and $w^\enn=s_{i_1}^\enn\cdots s_{i_\ell}^\enn$ be as in the theorem. Note that for any $p\in\mathcal{B}(\vec{w})$, 
\begin{equation*}
    \Psi^\enn(\textnormal{end}(p))\lambda=n\textnormal{wt}(p)+\phi(p)\lambda,
\end{equation*}
so 
\begin{equation*}
    \{\Psi^\enn(\textnormal{end}(p))\lambda: p\in\mathcal{B}(\vec{w})\}
\end{equation*}
is the set of powers of $x$ appearing in (\ref{result}). Further notice that for all $p\in \mathcal{B}(\vec{w})$, the elements $\Psi^\enn(\textnormal{end}(p))\lambda$ 
have the form
\begin{equation}\label{subword}
    s_{i_{j_1}}^\enn s_{i_{j_2}}^\enn\cdots s_{i_{j_b}}^\enn \lambda
\end{equation}
for $1\leq j_1<j_2\cdots <j_b\leq \ell$. (See Remark~\ref{alcove walk explanation}.) This leads us to define a partial order $\overset{ n}{\leq }$ on $\mathbb{Z}^r$ such that $\nu\overset{ n}{\leq } \mu$ if and only if $\nu$ has the form (\ref{subword}). 

Here is an equivalent characterization. Given $\mu,\nu\in\mathbb{Z}^r$, let $w_\mu^\enn, w_\nu^\enn\in W_{\textnormal{Cox}}^\enn$ and $\lambda_{\mu},\lambda_{\nu}\in A^\enn$ be defined by
\begin{equation}\label{mu and nu}
    w_\mu^\enn \lambda_{\mu}=\mu\textnormal{ and }w_\nu^\enn \lambda_{\nu}=\nu,
\end{equation}
where $w_\mu^\enn$ and $w_\nu^\enn$ are taken to be the unique shortest elements such that (\ref{mu and nu}) holds. Then $\nu\overset{ n}{\leq } \mu$ if and only if $\lambda_\nu=\lambda_\mu$ and $w_\nu^\enn\leq w_\mu^\enn$ in the Bruhat order on the Coxeter group $W_{\textnormal{Cox}}^\enn$. 
(Recall Definition~\ref{bruhat def} and Lemma~\ref{bruhat subwords}.) The fact that $\overset{ n}{\leq }$ is a partial order on $\mathbb{Z}^r$ follows from the fact that the Bruhat order is a partial order on $W_{\textnormal{Cox}}^\enn$. 

We now have the following result, which precisely characterizes the powers of $x$ appearing in $E_\mu^\enn$. 

\begin{corollary}\label{triangularity corollary}
For $\mu\in\mathbb{Z}^r$, 
\begin{equation}\label{triangular form}
    E_\mu^\enn =x^\mu+\sum_{\nu\overset{ n}{< }\mu} c_\nu x^{\nu}
\end{equation}
for some scalars $c_\nu\in\mathbb{F}^\enn$.
In particular, the family $\{E_\mu^\enn:\mu\in\mathbb{Z}^r\}$ is a basis for $\mathbb{F}^\enn[x^{\pm 1}]$ over $\mathbb{F}^\enn$.

Further, the scalars $c_\nu$ are all nonzero. 
\end{corollary}
\begin{proof}
The first statement, triangularity with respect to the order $\overset{ n}{\leq }$ on $\mathbb{Z}^r$, follows from Theorem~\ref{main} and the above discussion. 

To see that the scalars $c_\nu$ are all nonzero, we observe that the coefficients appearing in \eqref{result} are all strictly positive if the parameters are appropriately specialized. In particular, let $0<k<1$ and let $G_s^\enn =1$ for all $s\not\equiv 0\textnormal{ mod }n$. Choose $0<q<1$ sufficiently small that the finitely many expressions $\gamma(-\beta_1^\enn;\lambda),\dots, \gamma(-\beta_\ell^\enn;\lambda)$ satisfy
\begin{equation*}
    0<\gamma(-\beta_j^\enn;\lambda)<1.
\end{equation*}
Then every term in \eqref{result} is visibly positive. 
\end{proof}

Looking at examples of $E_\mu^\enn$ for fixed $\mu\in\mathbb{Z}^r$ and varying $n\geq 1$ (see Appendix~\ref{appendix: nonsymmetric}), one notices that the polynomials have fewer terms as $n$ increases. We now make this observation precise. First, we must prove some results about the partial order $\overset{ n}{\leq }$. 

We may explicitly characterize the partial order $\overset{ n}{\leq }$ as follows. This result is known for $n=1$ and follows similarly for the general case. We provide a proof from the perspective of alcove walks. 

\begin{lemma}\label{Lemma: ordering}
Let $\mu\in\mathbb{Z}^r$ and $\widehat{\alpha}\in\widetilde{\Phi}_+^\enn$. We have
\begin{equation}\label{reflection inequality}
    s_{\widehat{\alpha}}\mu \overset{ n}{< } \mu
\end{equation}
if and only if
\begin{equation}\label{positive side}
    \langle \widehat{\alpha},\mu\rangle < 0. 
\end{equation}
\end{lemma}
\begin{proof}
First, we note that we only need to prove (\ref{positive side}) implies (\ref{reflection inequality}): if $\langle \widehat{\alpha},\mu\rangle=0$ then $s_{\widehat{\alpha}}\mu = \mu$, and if $\langle \widehat{\alpha},\mu\rangle>0$ then $\langle \widehat{\alpha},s_{\widehat{\alpha}}\mu\rangle<0$ by (\ref{self adjoint}). 

We will instead prove the corresponding result for $\mathbb{R}^r$. By abuse of notation, for $v\in\mathbb{R}^r$, let $\ell(v)$ be the length of a shortest alcove walk from $1$ to the alcove containing $v$. We will show that for $v\in \mathbb{R}^r$ and $\widehat{\alpha}\in\widetilde{\Phi}_+$,
\begin{equation}\label{positive side 2}
    \langle \widehat{\alpha}, v \rangle < 0 
\end{equation}
implies
\begin{equation}\label{length shorter}
    \ell(s_{\widehat{\alpha}}v)< \ell(v).
\end{equation}
The desired result will follow by taking $v\in \frac{1}{n}\mathbb{Z}^r\subseteq \mathbb{R}^r$ and applying the isomorphisms $\Psi^\enn$ to everything in sight (recall (\ref{psi compatibility})). 

Let $v$ and $\widehat{\alpha}$ be as above and assume (\ref{positive side 2}) holds. Let $p$ be an alcove walk from $1$ to the alcove containing $v$ of the shortest length $\ell(v)$. Note that, for $u$ in the fundamental alcove $1$, we have $\langle \widehat{\alpha},u\rangle \geq 0$ (recall the definition (\ref{fundamental alcove})), so the alcove walk $p$ must pass through the hyperplane $\mathfrak{h}^{\widehat{\alpha}}$ at some step $j$. Let $p'$ be the alcove walk obtained from $p$ by folding at step $j$ (recall the definition and discussion before Remark~\ref{alcove walk explanation}). Then $p'$ has a shorter length than $p$, and by (\ref{folding ends}),
\begin{equation*}
    \textnormal{end}(p')=s_{\widehat{\alpha}}\textnormal{end}(p)=s_{\widehat{\alpha}} w.
\end{equation*}
Then we have (\ref{length shorter}). 
\end{proof}

We now observe
\begin{proposition}\label{n implies m}
Let $m$ be a positive integer with $m|n$ and let $\mu,\nu\in\mathbb{Z}^r$. If $\nu\overset{ n}{\leq }\mu$, then $\nu\overset{ m}{\leq }\mu$. 
\end{proposition}
\begin{proof}
We may reduce to the case $\nu=s_{\widehat{\alpha}}\mu$ for $\widehat{\alpha}\in\Phi^\enn_+$. Write 
\begin{equation*}
    \widehat{\alpha}=n\alpha+sn^2\delta
\end{equation*}
for $\alpha\in\Phi_+$ and $s\in\mathbb{Z}$, and let $c=\frac{n}{m}\in\mathbb{Z}$. 
By Lemma~\ref{Lemma: ordering}, we have 
\begin{equation*}
    0>\langle \widehat{\alpha},\mu\rangle=n(\alpha,\mu)+sn^2=cm(\alpha,\mu)+sc^2 m^2=c\langle m\alpha+scm^2\delta,\mu\rangle.
\end{equation*}
Dividing by the positive integer $c$, we see that
\begin{equation*}
    \langle m\alpha+scm^2\delta,\mu\rangle<0.
\end{equation*}
Since $\widehat{\alpha}$ is a positive root in $\widetilde{\Phi}^\enn$, we see that $m\alpha+scm^2\delta$ is a positive root in $\widetilde{\Phi}^{(m)}$. Further, 
\begin{equation*}
    s_{\widehat{\alpha}}=s_{m\alpha+scm^2\delta}.
\end{equation*}
Then by Lemma~\ref{Lemma: ordering}, 
\begin{equation*}
    \nu=s_{m\alpha+scm^2\delta}\mu \overset{ m}{\leq }\mu
\end{equation*}
\end{proof}




By Corollary~\ref{triangularity corollary} and Proposition~\ref{n implies m}, we have

\begin{corollary}\label{macdonald comparison}
Let $m$ be a positive integer with $m|n$ and let $\mu,\nu\in\mathbb{Z}^r$. 
If $x^\nu$ appears with nonzero coefficient in $E_\mu^\enn$, then $x^\nu$ appears with nonzero coefficient in $E_\mu^{(m)}$. In particular, this holds for the case $m=1$, where $E_\mu^{(1)}=E_\mu$ is the Macdonald polynomial. 
\end{corollary}


Finally, for $u\in W_0$ and $\mu\in\mathbb{Z}^r$, we record an alcove walk formula for $T_u E_\mu^\enn$. These polynomials are the natural generalization of the permuted basement nonsymmetric Macdonald polynomials of \cite{permutedBasement2, permutedBasement}. The proof is the same as the proof of Theorem~\ref{main}, except we apply Theorem~\ref{ry main} to $T_u \tau_w^\vee$ rather than $\tau_w^\vee$. 

\begin{theorem}\label{permuted basement}
Let all notation be as in Theorem~\ref{main}, and let $u\in W_0$. Then
\begin{align*}
    T_u E_\mu^\enn\sim \pi^\enn(T_u \tau_w^\vee) x^\lambda =&\sum_{p\in \mathcal{B}(u,\vec{w})}  x^{n\textnormal{wt}(p)+\phi(p)\lambda}\left (\prod_{a=1}^{\ell(\phi(p))}\sigma((\lambda,s_{u_{p,t}}\cdots s_{u_{p,a+1}} \alpha_{u_{p,a}})) \right)\cdot \notag
    \\& \cdot\left (\prod_{j\in f^+(p)}\dfrac{k^{-1}-k}{1-\gamma({-\beta_j^\enn};\lambda)} \right) \left(\prod_{j\in f^-(p)}\dfrac{(k^{-1}-k)\gamma({-\beta_j^\enn};\lambda)}{1-\gamma({-\beta_j^\enn};\lambda)} \right).
\end{align*}
\end{theorem}

Note that the only difference between Theorem~\ref{main} and the slightly more general Theorem~\ref{permuted basement} is that we sum over $\mathcal{B}(u,\vec{w})$ rather than $\mathcal{B}(\vec{w})=\mathcal{B}(1,\vec{w}).$ (See Section~\ref{section: alcove walks} for the relevant definitions.) 

\section{Symmetric SSV polynomials}\label{section: symmetric}

We will now construct a generalization of the symmetric Macdonald polynomials, establish some basic properties, and give a combinatorial formula. In general, these polynomials will \emph{not} by symmetric with respect to the usual action of $W_0$, but rather (a conjugate of) the Weyl group action of Chinta and Gunnells, which depends on the metaplectic degree $n$: see \cite{cg1,cg2} for details. 

Consider the Hecke symmetrizer
\begin{equation*}
    U=\sum_{u\in W_0}k^{\ell(u)}T_u\in\mathbb{H}^\enn.
\end{equation*}
For convenience, we will also denote $\pn(U)=U$. 
In the case $n=1$, for $\lambda\in\mathbb{Z}^r$ dominant and any $\mu\in W_0 \lambda$, $UE_\mu$ is a nonzero scalar multiple of the symmetric Macdonald polynomial $P_\lambda$. (See \cite{mac2}.) This motivates the following definition. For $\mu\in\mathbb{Z}^r$, let
\begin{equation*}
    P_\mu^\enn=U E_\mu^\enn.
\end{equation*}
(Note that for $\lambda$ dominant, we may have $P_\lambda^{(1)}\neq P_\lambda$ simply because of normalization: we have not required that the coefficient of $x^\mu$ be $1$, since we have not proven at this point that there is a nonzero $x^\mu$ term.) 

\begin{lemma}
Let $f(Y)$ be a symmetric Laurent polynomial in $Y^{n\epsilon_1},\dots, Y^{n\epsilon_{r}}$. Then for $1\leq i\leq r-1$,
\begin{equation*}
    f(Y)T_i=T_if(Y).
\end{equation*}
In particular, $f(Y)$ commutes with $U$. 
\end{lemma}
\begin{proof}
Recall the Bernstein-Zelevinsky relations (\ref{BZ Y}). Summing over all $Y^\mu$ appearing in $f(Y)$, we get
\begin{equation*}
    T_i f(Y)-(s_i f)(Y)T_i=(k-k^{-1})\left( \dfrac{f(Y)-(s_i f)(Y)}{1-Y^{-\alpha_i^\enn}}  \right)=0
\end{equation*}
by the symmetry of $f$. 
\end{proof}

Then we immediately have
\begin{corollary}
For any $\mu\in\mathbb{Z}^r$ and any symmetric Laurent polynomial $f(Y)$ in $Y^{n\epsilon_1},\dots, Y^{n\epsilon_{r}}$, 
$P_\mu$ and $E_\mu$ are eigenfunctions of $f(Y)$ with the same eigenvalue. 
\end{corollary}

The proof of the following lemma is simple and can be found in \cite{mac2}. 

\begin{lemma}\label{T is k}
For $1\leq i\leq r-1$, 
\begin{equation*}
    (T_i-k) U=U(T_i-k)=0.
\end{equation*}
\end{lemma}

Following the arguments in \cite{mac2}, we can also prove:

\begin{proposition}
Let $1\leq i \leq r-1$ and $\mu\in \mathbb{Z}^r$. Then 
\begin{equation*}
    P_{s_i^\enn \mu}^\enn\sim P_\mu^\enn. 
\end{equation*}
\end{proposition}
\begin{proof}
If $s_i^\enn \mu=\mu$, then there is nothing to prove, so let us assume that 
\begin{equation*}
    s_i^\enn \mu\neq \mu.
\end{equation*}
By Proposition~\ref{intertwiner property}, for some nonzero scalars $c$ and $c'$, 
\begin{equation*}
    cE_{s_i^\enn \mu}^\enn=\pn(\uptau_i^\vee) E_\mu^\enn=(\pn(T_i)+c')E_\mu^\enn,
\end{equation*}
so
\begin{equation*}
    \pn(T_i) E_\mu^\enn=cE_{s_i^\enn \mu}^\enn-c' E_\mu^\enn.
\end{equation*}
Then
\begin{align*}
    k P_\mu^\enn &=k U E_\mu^\enn 
    \\&= U k E_\mu^\enn 
    \\&= U \pn(T_i) E_\mu^\enn \textnormal{(by (\ref{T is k}))}
    \\&=U(cE_{s_i^\enn \mu}^\enn-c' E_\mu^\enn)
    \\&=c P_{s_i^\enn \mu}^\enn -c' P_\mu^\enn, 
\end{align*}
implying
\begin{equation*}
    c^{-1}(c'+k)P_\mu^\enn=P_{s_i^\enn \mu}^\enn.
\end{equation*}

We must show that the scalar $c^{-1}(c'+k)$ is nonzero. Recalling the definition (\ref{intertwiner def}) of $\uptau_i^\vee$, we see that
\begin{equation}\label{c'}
    c'=\dfrac{k^{-1}-k}{1-\gamma({-n\alpha_i};\mu)}.
\end{equation}
Since $s_i^\enn \mu\neq \mu$, by (\ref{gamma clear q}) we have
\[ \gamma({-n\alpha_i};\mu)=\gamma({n\alpha_i};\mu)^{-1}\neq k^{-2} \]
(just as in the proof of Proposition~\ref{intertwiner property}). Then
(\ref{c'}) is not equal to $-k$ and $c^{-1}(c'+k)\neq 0$.
\end{proof}
Therefore, it suffices to consider $P_\mu^\enn$ where $\mu$ is dominant. 

By \cite{ssv} Theorem 3.21 and Proposition 4.2, there exists a representation 
\begin{equation*}
    \overline{\sigma}:W_0\rightarrow \textnormal{GL}(\mathbb{F}^\enn [x^{\pm 1}])
\end{equation*}
of $W_0$ such that for $1\leq i\leq r-1$ and $\mu\in\mathbb{Z}^r$,
\begin{equation}\label{CG property}
    (\pi^\enn(T_i)-k)x^\mu=k^{-1}\dfrac{1-k^2 x^{n\alpha_i}}{1-x^{n\alpha_i}}(\overline{\sigma}(s_i)-1)x^\mu.
\end{equation}
Specifically, in the notation of \cite{ssv}, $\overline{\sigma}$ is defined by
\begin{equation*}
    \overline{\sigma}(s_i)=x^{\rho^n-\rho}\sigma(s_i)x^{\rho-\rho^n}.
\end{equation*}
The map $\sigma$ in the notation of \cite{ssv} is the Weyl group action of Chinta and Gunnells \cite{cg1,cg2}. 

\begin{proposition}
For $\mu\in\mathbb{Z}^r$ dominant, $P_\mu^\enn$ is symmetric with respect to the representation $\overline{\sigma}$: for $1\leq i\leq r-1$,
\begin{equation*}
    \overline{\sigma} (s_i)P_\mu^\enn=P_\mu^\enn.
\end{equation*}
\end{proposition}
\begin{proof}
By (\ref{CG property}) and Lemma~\ref{T is k},
\begin{equation*}
    \overline{\sigma} (s_i)P_\mu^\enn-P_\mu^\enn=k\dfrac{1-x^{n\alpha_i}}{1-k^2 x^{n\alpha_i}}(\pi^\enn(T_i)-k)P_\mu^\enn=0.
\end{equation*}
\end{proof}

We now have a family of polynomials 
\[\{P_\mu^\enn: \mu\in \mathbb{Z}^r, \mu\textnormal{ dominant}\}\] 
that is symmetric under an action of $W_0$ and that generalizes the symmetric Macdonald polynomials of \cite{mac1}. We call this family the \emph{symmetric SSV polynomials}. We would like to emphasize, however, that for $n\neq 1$, the symmetric SSV polynomials are \emph{not} in general symmetric with respect to the usual action of $W_0$. (See Appendix~\ref{appendix: symmetric} below for examples.) 

We will now give a combinatorial formula for the symmetric SSV polynomials. The idea is the same as in the proof of Theorem~3.4 in \cite{ry}. 

\begin{theorem}\label{thm: symmetric}
Let $\mu\in \mathbb{Z}^r$ be dominant, and otherwise let all notation be as in Theorem~\ref{main}. 
Then
\begin{align*}
   P_\mu^\enn\sim\sum_{u\in W_0}k^{\ell(u)}\sum_{p\in \mathcal{B}(u,\vec{w})} & x^{n\textnormal{wt}(p)+\phi(p)\lambda}\left (\prod_{a=1}^{\ell(\phi(p))}\sigma((\lambda,s_{u_{p,t}}\cdots s_{u_{p,a+1}} \alpha_{u_{p,a}})) \right)\cdot \notag
    \\& \cdot\left (\prod_{j\in f^+(p)}\dfrac{k^{-1}-k}{1-\gamma({-\beta_j^\enn};\lambda)} \right) \left(\prod_{j\in f^-(p)}\dfrac{(k^{-1}-k)\gamma({-\beta_j^\enn};\lambda)}{1-\gamma({-\beta_j^\enn};\lambda)} \right).
\end{align*}
\end{theorem}
\begin{proof}
We have
\begin{equation*}
    P_\mu^\enn=\sum_{u\in W_0}k^{\ell(u)} \pn(T_u) E_\mu^\enn\sim
    \sum_{u\in W_0}k^{\ell(u)} \pn(T_u \uptau_w^\vee) x^\lambda
\end{equation*}
by Proposition~\ref{fundamental domain}.
Then the result follows from Theorem~\ref{permuted basement}. 
\end{proof}


\section{\texorpdfstring{$q$}{q}-limits}\label{section: q limits}
In this section, we will record formulas for the limits
\begin{equation*}
    E_\mu^\enn(x;0,k)=\lim_{q\rightarrow 0}E_\mu^\enn(x;q,k)
\end{equation*}
and 
\begin{equation*}
    E_\mu^\enn(x;\infty,k)=\lim_{q\rightarrow \infty}E_\mu^\enn(x;q,k).
\end{equation*}
We will also record formulas for $P_\mu^\enn(x;0,k)$ and $P_\mu^\enn(x;\infty,k)$ and derive a simple positivity result. 

\begin{remark}
In \cite{ssv2}, $E_\mu^\enn(x;\infty,k)$ is related to metaplectic Iwahori-Whittaker functions (after specializing the parameters $G_i^\enn$ to particular Gauss sums). See \cite{pp} for the definition of the relevant metaplectic Iwahori-Whittaker functions. To see the relationship between Macdonald polynomials and Iwahori Whittaker functions in the non-metaplectic case, see \cite{BBL}. 
\end{remark}

The following result is a consequence of Lemma~\ref{q power}, (\ref{result}), and basic calculus. 

\begin{theorem}\label{main q limit}
Let all notation be as in Theorem~\ref{main}. The limits $E_\mu^\enn(x;0,k)$ and $E_\mu^\enn(x,\infty,k)$ are well-defined. We have 
\begin{align*}
    &\left (\prod_{a=1}^{\ell(\phi(\widetilde{p}))}\sigma((\lambda,s_{u_{\widetilde{p},t}}\cdots s_{u_{\widetilde{p},a+1}} \alpha_{u_{\widetilde{p},a}})) \right)E_\mu^\enn(x;0,k)\notag
    \\&=\sum_{p\in \mathcal{B}^+(\vec{w})}  x^{n\textnormal{wt}(p)+\phi(p)\lambda}\left (\prod_{a=1}^{\ell(\phi(p))}\sigma((\lambda,s_{u_{p,t}}\cdots s_{u_{p,a+1}} \alpha_{u_{p,a}})) \right)(k^{-1}-k)^{|f^+(p)|}
\end{align*}
and
\begin{align*}
    &\left (\prod_{a=1}^{\ell(\phi(\widetilde{p}))}\sigma((\lambda,s_{u_{\widetilde{p},t}}\cdots s_{u_{\widetilde{p},a+1}} \alpha_{u_{\widetilde{p},a}})) \right)E_\mu^\enn(x;\infty,k)\notag
    \\&=\sum_{p\in \mathcal{B}^-(\vec{w})}  x^{n\textnormal{wt}(p)+\phi(p)\lambda}\left (\prod_{a=1}^{\ell(\phi(p))}\sigma((\lambda,s_{u_{p,t}}\cdots s_{u_{p,a+1}} \alpha_{u_{p,a}})) \right)(k-k^{-1})^{|f^-(p)|},
\end{align*}
where $\mathcal{B}^+(\vec{w})$ (resp. $\mathcal{B}^-(\vec{w})$) is the set of all alcove walks in $\mathcal{B}(\vec{w})$ with only positive (resp. negative) folds and $|f^{+}(p)|$ (resp. $|f^-(p)|$) is the number of positive (resp. negative) folds. 
\end{theorem}


The same observations also give combinatorial formulas for the $q$-limits of the symmetric SSV polynomials $P_\mu^\enn$:

\begin{theorem}\label{symmetric q limit}
Let $\mu\in \mathbb{Z}^r$ be dominant, and otherwise let all notation be as in Theorem~\ref{main}. 
Then
\begin{align*}
   &P_\mu^\enn(x;0,k)\notag
   \\&=\sum_{u\in W_0}k^{\ell(u)}\sum_{p\in \mathcal{B}^+(u,\vec{w})} & x^{n\textnormal{wt}(p)+\phi(p)\lambda}\left (\prod_{a=1}^{\ell(\phi(p))}\sigma((\lambda,s_{u_{p,t}}\cdots s_{u_{p,a+1}} \alpha_{u_{p,a}})) \right)(k^{-1}-k)^{|f^+(p)|} 
\end{align*}
and 
\begin{align*}
   &P_\mu^\enn(x;\infty,k)\notag
   \\&=\sum_{u\in W_0}k^{\ell(u)}\sum_{p\in \mathcal{B}^-(u,\vec{w})} & x^{n\textnormal{wt}(p)+\phi(p)\lambda}\left (\prod_{a=1}^{\ell(\phi(p))}\sigma((\lambda,s_{u_{p,t}}\cdots s_{u_{p,a+1}} \alpha_{u_{p,a}})) \right)(k-k^{-1})^{|f^-(p)|},
\end{align*}
where $\mathcal{B}^+(u,\vec{w})$ (resp. $\mathcal{B}^-(u,\vec{w})$) is the set of all alcove walks in $\mathcal{B}(u,\vec{w})$ with only positive (resp. negative) folds. 
\end{theorem}

Since $\sigma(a)\in\{k,k^{-1}, G_a\}$ for all $a\in\mathbb{Z}$, Theorems \ref{main q limit} and \ref{symmetric q limit} immediately imply:
\begin{corollary}
Let $\mu,\nu\in\mathbb{Z}^r$. The coefficient of $x^\nu$ in $E_\mu^\enn(x; 0,k)$, $E_\mu^\enn(x; \infty,k)$, $P_\mu^\enn(x; 0,k)$, and $P_\mu^\enn(x;\infty,k)$ is a Laurent polynomial in $k, G_1^\enn,\dots, G^\enn_{\lfloor n/2\rfloor}$. If we specialize $k$ to be a real number with $0<k\leq 1$ (respectively $k\geq 1$), then the coefficient of $x^\nu$ in $E_\mu^\enn(x; 0,k)$ and $P_\mu^\enn(x; 0,k)$ (respectively $E_\mu^\enn(x; \infty,k)$ and $P_\mu^\enn(x; \infty,k)$) is a Laurent polynomial in $G_1^\enn,\dots, G^\enn_{\lfloor n/2\rfloor}$ \emph{with nonnegative coefficients}. 
\end{corollary}

We have the following corollary of Theorem~\ref{main q limit}.

\begin{corollary}
If $\mu\in\mathbb{Z}^r$ is dominant, then
\begin{equation*}
    E_\mu^\enn(x;0,k)= x^\mu.
\end{equation*}
If $\mu\in\mathbb{Z}^r$ is antidominant ($\mu_i<\mu_j$ for $i<j$), then
\begin{equation*}
    E_\mu^\enn(x;\infty,k)= x^\mu.
\end{equation*}
\end{corollary}
\begin{proof}
Assume $\mu$ is dominant: the antidominant case is similar. As usual, take $\lambda\in A^\enn$ and the shortest $w^\enn\in W^\enn_{\textnormal{Cox}}$ such that $w^\enn \lambda=\mu$, and write $w=(\Psi^\enn)^{-1}(w^\enn)$. 
As in Remark~\ref{alcove walk explanation}, we view $\ell(w)$ as the number of hyperplanes of the form $\mathfrak{h}^{\alpha+s\delta}$ crossed by a shortest unfolded alcove walk $\widetilde{p}$ from $1$ to $w$. Since $(\mu,\alpha)>0$ for all $\alpha\in\Phi_+$, every crossing of a hyperplane $\mathfrak{h}^{\alpha+s\delta}$ (in this case, $s\in\mathbb{Z}$ with $s>0$) will be from the negative side to the positive side. (If not, some hyperplane is crossed twice and $\widetilde{p}$ is not a shortest alcove walk from $1$ to $w$.) Then for any folded alcove walk of the same type as $\widetilde{p}$, the first fold must be a negative fold. This implies that $\mathcal{B}^+(\vec{w})=\{\widetilde{p}\}$. 
%
\end{proof}

\appendix
\section{Examples (nonsymmetric)}\label{appendix: nonsymmetric}

Here we give some examples of SSV polynomials $E_\mu^\enn$, computed using the combinatorial formula of Theorem~\ref{main} (with no terms combined). We simplify the terms slightly, using the fact that $G_{n/2}^\enn=\pm 1$, and replacing the notation $G_a^\enn$ with the cleaner notation $G_a$ for $a\in\mathbb{Z}$. 

The reader should compare these expressions with the Appendix of \cite{ssv}, where the same polynomials are calculated. The expressions given here simplify to the ones in \cite{ssv}, modulo the substitutions $q\rightarrow q^n$ and $G_a=-kg_{-a}$ for $a\in\mathbb{Z}$. 

{\small

\begin{flalign*}
    E_{(0,0,0)}^{(1)}  &=1
    &\\E_{(0,0,0)}^{(2)}  &=1
    &\\E_{(0,0,0)}^{(3)}  &=1
    &\\E_{(0,0,0)}^{(4)}  &=1
    &\\E_{(0,0,0)}^{(5)}  &=1
\end{flalign*}
    \begin{flalign*}
    E_{(1,0,0)}^{(1)}  &={  x_1}
    &\\E_{(1,0,0)}^{(2)}  &={  x_1}
    &\\E_{(1,0,0)}^{(3)}  &={  x_1}
    &\\E_{(1,0,0)}^{(4)}  &={  x_1}
    &\\E_{(1,0,0)}^{(5)}  &={  x_1}
    \end{flalign*}
    
    \begin{flalign*}
    E_{(0,1,0)}^{(1)}  &= { x_2}+{\frac {{ x_1}\, \left( k-1 \right)  \left( k+1 \right) }{{k}^{4}q-1}}
    &\\E_{(0,1,0)}^{(2)}  &={ x_2}+{\frac {{ x_1}\, \left( k-1 \right)  \left( k+1 \right) }{k \left( -{ G_{{1}}}+kq \right) }}
    &\\E_{(0,1,0)}^{(3)}  &={ x_2}+{\frac {{ x_1}\, \left( k-1 \right)  \left( k+1 \right) { G_{{1}}}^{2}}{k \left( -{ G_{{1}}}^{3}+kq \right) }}
    &\\E_{(0,1,0)}^{(4)}  &={ x_2}+{\frac {{ x_1}\, \left( k-1 \right)  \left( k+1 \right) { G_{{1}}}^{2}}{k \left( -{ G_{{1}}}^{3}+kq \right) }}
    &\\E_{(0,1,0)}^{(5)}  &={ x_2}+{\frac {{ x_1}\, \left( k-1 \right)  \left( k+1 \right) { G_{{1}}}^{2}}{k \left( -{ G_{{1}}}^{3}+kq \right) }}
    \end{flalign*}
    
    \begin{flalign*}
    E_{(0,0,1)}^{(1)}&={ x_3}+{\frac {{ x_2}\, \left( k-1 \right)  \left( k+1 \right) }{q{k}^{2}-1}}+{\frac {{ x_1}\, \left( k-1 \right)  \left( k+1 \right) {k}^{2}}{{k}^{4}q-1}}+{\frac {{ x_1}\, \left( k-1 \right) ^{2} \left( k+1 \right) ^{2}}{ \left( q{k}^{2}-1 \right)  \left( {k}^{4}q-1 \right) }}
    &\\E_{(0,0,1)}^{(2)}&={ x_3}-{\frac {{ x_2}\, \left( k-1 \right)  \left( k+1 \right) }{k{ G_{{1}}}-q}}+{\frac {{ x_1}\, \left( k-1 \right)  \left( k+1 \right)  G_{{1}}}{-{ G_{{1}}}+kq}}-{\frac {{ x_1}\, \left( k-1 \right) ^{2} \left( k+1 \right) ^{2}}{k \left( k{ G_{{1}}}-q \right)  \left( -{ G_{{1}}}+kq \right) }}
    &\\E_{(0,0,1)}^{(3)}&={ x_3}-{\frac {{ x_2}\, \left( k-1 \right)  \left( k+1 \right) { G_{{1}}}^{2}}{k{ G_{{1}}}^{3}-q}}+{\frac {{ x_1}\, \left( k-1 \right)  \left( k+1 \right)  G_{{1}}}{-{ G_{{1}}}^{3}+kq}}-{\frac {{ x_1}\, \left( k-1 \right) ^{2} \left( k+1 \right) ^{2}{ G_{{1}}}^{4}}{k \left( k{ G_{{1}}}^{3}-q \right)  \left( -{ G_{{1}}}^{3}+kq \right) }}
    &\\E_{(0,0,1)}^{(4)}&={ x_3}-{\frac {{ x_2}\, \left( k-1 \right)  \left( k+1 \right) { G_{{1}}}^{2}}{k{ G_{{1}}}^{3}-q}}+{\frac {{ x_1}\, \left( k-1 \right)  \left( k+1 \right)  G_{{1}}}{-{ G_{{1}}}^{3}+kq}}-{\frac {{ x_1}\, \left( k-1 \right) ^{2} \left( k+1 \right) ^{2}{ G_{{1}}}^{4}}{k \left( k{ G_{{1}}}^{3}-q \right)  \left( -{ G_{{1}}}^{3}+kq \right) }}
    &\\E_{(0,0,1)}^{(5)}&={ x_3}-{\frac {{ x_2}\, \left( k-1 \right)  \left( k+1 \right) { G_{{1}}}^{2}}{k{ G_{{1}}}^{3}-q}}+{\frac {{ x_1}\, \left( k-1 \right)  \left( k+1 \right)  G_{{1}}}{-{ G_{{1}}}^{3}+kq}}-{\frac {{ x_1}\, \left( k-1 \right) ^{2} \left( k+1 \right) ^{2}{ G_{{1}}}^{4}}{k \left( k{ G_{{1}}}^{3}-q \right)  \left( -{ G_{{1}}}^{3}+kq \right) }}
    \end{flalign*}
    
\begin{flalign*}
    E_{(0,1,1)}^{(1)} &= { x_2}\,{ x_3}+{\frac {{ x_1}\,{ x_3}\, \left( k-1 \right)  \left( k+1 \right) }{q{k}^{2}-1}}+{\frac {{ x_1}\,{ x_2}\, \left( k-1 \right)  \left( k+1 \right) {k}^{2}}{q{k}^{4}-1}}+{\frac {{ x_1}\,{ x_2}\, \left( k-1 \right) ^{2} \left( k+1 \right) ^{2}}{ \left( q{k}^{2}-1 \right)  \left( q{k}^{4}-1 \right) }}
    &\\E_{(0,1,1)}^{(2)} &= { x_2}\,{ x_3}-{\frac {{ x_1}\,{ x_3}\, \left( k-1 \right)  \left( k+1 \right) }{k{ G_{{1}}}-q}}+{\frac {{ x_1}\,{ x_2}\, \left( k-1 \right)  \left( k+1 \right)  G_{{1}}}{-{ G_{{1}}}+kq}}-{\frac {{ x_1}\,{ x_2}\, \left( k-1 \right) ^{2} \left( k+1 \right) ^{2}}{k \left( k{ G_{{1}}}-q \right)  \left( -{ G_{{1}}}+kq \right) }}
    &\\E_{(0,1,1)}^{(3)} &={ x_2}\,{ x_3}-{\frac {{ x_1}\,{ x_3}\, \left( k-1 \right)  \left( k+1 \right) { G_{{1}}}^{2}}{k{ G_{{1}}}^{3}-q}}+{\frac {{ x_1}\,{ x_2}\, \left( k-1 \right)  \left( k+1 \right)  G_{{1}}}{-{ G_{{1}}}^{3}+kq}}-{\frac {{ x_1}\,{ x_2}\, \left( k-1 \right) ^{2} \left( k+1 \right) ^{2}{ G_{{1}}}^{4}}{k \left( k{ G_{{1}}}^{3}-q \right)  \left( -{ G_{{1}}}^{3}+kq \right) }}
    &\\E_{(0,1,1)}^{(4)} &= { x_2}\,{ x_3}-{\frac {{ x_1}\,{ x_3}\, \left( k-1 \right)  \left( k+1 \right) { G_{{1}}}^{2}}{k{ G_{{1}}}^{3}-q}}+{\frac {{ x_1}\,{ x_2}\, \left( k-1 \right)  \left( k+1 \right)  G_{{1}}}{-{ G_{{1}}}^{3}+kq}}-{\frac {{ x_1}\,{ x_2}\, \left( k-1 \right) ^{2} \left( k+1 \right) ^{2}{ G_{{1}}}^{4}}{k \left( k{ G_{{1}}}^{3}-q \right)  \left( -{ G_{{1}}}^{3}+kq \right) }}
    &\\E_{(0,1,1)}^{(5)} &= { x_2}\,{ x_3}-{\frac {{ x_1}\,{ x_3}\, \left( k-1 \right)  \left( k+1 \right) { G_{{1}}}^{2}}{k{ G_{{1}}}^{3}-q}}+{\frac {{ x_1}\,{ x_2}\, \left( k-1 \right)  \left( k+1 \right)  G_{{1}}}{-{ G_{{1}}}^{3}+kq}}-{\frac {{ x_1}\,{ x_2}\, \left( k-1 \right) ^{2} \left( k+1 \right) ^{2}{ G_{{1}}}^{4}}{k \left( k{ G_{{1}}}^{3}-q \right)  \left( -{ G_{{1}}}^{3}+kq \right) }}
\end{flalign*}
\begin{flalign*}
    E_{(1,0,1)}^{(1)} &= x_1 x_3 +{x_1 x_2 \frac { \left( k+1 \right)  \left( k-1 \right) }{q{k}^{4}-1}}
    &\\E_{(1,0,1)}^{(2)} &= x_1 x_3+{\frac {{ x_1}\,{ x_2}\, \left( k-1 \right)  \left( k+1 \right) }{k \left( -{ G_{{1}}}+kq \right) }}
    &\\E_{(1,0,1)}^{(3)} &= x_1 x_3+{\frac {{ x_1}\,{ x_2}\, \left( k-1 \right)  \left( k+1 \right) { G_{{1}}}^{2}}{k \left( -{ G_{{1}}}^{3}+kq \right) }}
    &\\E_{(1,0,1)}^{(4)} &= x_1 x_3+{\frac {{ x_1}\,{ x_2}\, \left( k-1 \right)  \left( k+1 \right) { G_{{1}}}^{2}}{k \left( -{ G_{{1}}}^{3}+kq \right) }}
    &\\E_{(1,0,1)}^{(5)} &= x_1 x_3+{\frac {{ x_1}\,{ x_2}\, \left( k-1 \right)  \left( k+1 \right) { G_{{1}}}^{2}}{k \left( -{ G_{{1}}}^{3}+kq \right) }}
\end{flalign*}
\begin{flalign*}
    E_{(1,1,0)}^{(1)} &= {  x_1}\,{  x_2}
    &\\E_{(1,1,0)}^{(2)} &= {  x_1}\,{  x_2}
    &\\E_{(1,1,0)}^{(3)} &= {  x_1}\,{  x_2}
    &\\E_{(1,1,0)}^{(4)} &= {  x_1}\,{  x_2}
    &\\E_{(1,1,0)}^{(5)} &= {  x_1}\,{  x_2}
\end{flalign*}
\begin{flalign*}
    E_{(2,0,0)}^{(1)} &= {{ x_1}}^{2}+{\frac {{ x_3}\,{ x_1}\, \left( k-1 \right)  \left( k+1 \right) q}{{k}^{2}q-1}}+{\frac {{ x_1}\,{ x_2}\, \left( k-1 \right)  \left( k+1 \right) {k}^{2}q}{q{k}^{4}-1}}+{\frac {{ x_1}\,{ x_2}\, \left( k-1 \right) ^{2} \left( k+1 \right) ^{2}q}{ \left( {k}^{2}q-1 \right)  \left( q{k}^{4}-1 \right) }}
    &\\E_{(2,0,0)}^{(2)} &= {{  x_1}}^{2}
    &\\E_{(2,0,0)}^{(3)} &= {{  x_1}}^{2}
    &\\E_{(2,0,0)}^{(4)} &= {{  x_1}}^{2}
    &\\E_{(2,0,0)}^{(5)} &= {{  x_1}}^{2}
\end{flalign*}
\begin{flalign*}
    E_{(0,2,0)}^{(1)} &= {{ x_2}}^{2}+{\frac {{{ x_1}}^{2} \left( k-1 \right)  \left( k+1 \right) }{ \left( {k}^{2}q-1 \right)  \left( {k}^{2}q+1 \right) }}+{\frac {{ x_2}\,{ x_3}\, \left( k-1 \right)  \left( k+1 \right) q}{{k}^{2}q-1}}+{\frac {{ x_3}\,{ x_1}\, \left( k-1 \right) ^{2} \left( k+1 \right) ^{2}q}{ \left( {k}^{2}q-1 \right) ^{2} \left( {k}^{2}q+1 \right) }}
    \\&+{\frac { \left( k+1 \right)  \left( k-1 \right) { x_2}\,{ x_1}}{{k}^{4}q-1}}+{\frac {{ x_1}\,{ x_2}\, \left( k-1 \right) ^{2} \left( k+1 \right) ^{2}{k}^{2}q}{ \left( {k}^{2}q-1 \right)  \left( {k}^{2}q+1 \right)  \left( {k}^{4}q-1 \right) }}+{\frac {{ x_1}\,{ x_2}\, \left( k-1 \right) ^{2} \left( k+1 \right) ^{2}{k}^{2}q}{ \left( {k}^{2}q-1 \right)  \left( {k}^{4}q-1 \right) }}
    \\&+{\frac {{ x_1}\,{ x_2}\, \left( k-1 \right) ^{3} \left( k+1 \right) ^{3}q}{ \left( {k}^{2}q-1 \right) ^{2} \left( {k}^{2}q+1 \right)  \left( {k}^{4}q-1 \right) }}
    &\\E_{(0,2,0)}^{(2)} &= {{ x_2}}^{2}+{\frac {{{ x_1}}^{2} \left( k-1 \right)  \left( k+1 \right) }{ \left( {k}^{2}q-1 \right)  \left( {k}^{2}q+1 \right) }}
    &\\E_{(0,2,0)}^{(3)} &= {{ x_2}}^{2}+{\frac {{{ x_1}}^{2} \left( k-1 \right)  \left( k+1 \right)  G_{{1}}}{k \left( k{q}^{2}{ G_{{1}}}^{3}-1 \right) }}
    &\\E_{(0,2,0)}^{(4)} &= {{ x_2}}^{2}+{\frac {{{ x_1}}^{2} \left( k-1 \right)  \left( k+1 \right) }{k \left( k{q}^{2}-{ G_{{2}}} \right) }}
    &\\E_{(0,2,0)}^{(5)} &= {{ x_2}}^{2}+{\frac {{{ x_1}}^{2} \left( k-1 \right)  \left( k+1 \right) { G_{{2}}}^{2}}{k \left( k{q}^{2}-{ G_{{2}}}^{3} \right) }}
\end{flalign*}
\begin{flalign*}
    E_{(0,0,2)}^{(1)} &= {{ x_3}}^{2}+{\frac {{{ x_2}}^{2} \left( k-1 \right)  \left( k+1 \right) }{ \left( kq-1 \right)  \left( kq+1 \right) }}+{\frac {{{ x_1}}^{2} \left( k-1 \right)  \left( k+1 \right) {k}^{2}}{ \left( q{k}^{2}-1 \right)  \left( q{k}^{2}+1 \right) }}+{\frac {{{ x_1}}^{2} \left( k-1 \right) ^{2} \left( k+1 \right) ^{2}}{ \left( kq-1 \right)  \left( kq+1 \right)  \left( q{k}^{2}-1 \right)  \left( q{k}^{2}+1 \right) }}
    \\&+{\frac {{ x_2}\,{ x_3}\, \left( k-1 \right)  \left( k+1 \right) }{q{k}^{2}-1}}+{\frac {{ x_2}\,{ x_3}\, \left( k-1 \right) ^{2} \left( k+1 \right) ^{2}q}{ \left( q{k}^{2}-1 \right)  \left( kq-1 \right)  \left( kq+1 \right) }}+{\frac {{ x_1}\,{ x_2}\, \left( k-1 \right) ^{2} \left( k+1 \right) ^{2}{k}^{2}q}{ \left( q{k}^{2}-1 \right) ^{2} \left( q{k}^{2}+1 \right) }}
    \\&+{\frac {{ x_3}\,{ x_1}\, \left( k-1 \right) ^{3} \left( k+1 \right) ^{3}q}{ \left( kq-1 \right)  \left( kq+1 \right)  \left( q{k}^{2}-1 \right) ^{2} \left( q{k}^{2}+1 \right) }}
    +{\frac {{ x_3}\,{ x_1}\, \left( k-1 \right)  \left( k+1 \right) {k}^{2}}{{k}^{4}q-1}}+{\frac {{ x_1}\,{ x_2}\, \left( k-1 \right) ^{2} \left( k+1 \right) ^{2}}{ \left( kq-1 \right)  \left( kq+1 \right)  \left( {k}^{4}q-1 \right) }}
    \\&+{\frac {{ x_3}\,{ x_1}\, \left( k-1 \right) ^{2} \left( k+1 \right) ^{2}{k}^{2}q}{ \left( q{k}^{2}-1 \right)  \left( q{k}^{2}+1 \right)  \left( {k}^{4}q-1 \right) }}
    +{\frac {{ x_1}\,{ x_2}\, \left( k-1 \right) ^{3} \left( k+1 \right) ^{3}{k}^{2}q}{ \left( kq-1 \right)  \left( kq+1 \right)  \left( q{k}^{2}-1 \right)  \left( q{k}^{2}+1 \right)  \left( {k}^{4}q-1 \right) }}
    \\&+{\frac {{ x_3}\,{ x_1}\, \left( k-1 \right) ^{2} \left( k+1 \right) ^{2}}{ \left( q{k}^{2}-1 \right)  \left( {k}^{4}q-1 \right) }}+{\frac {{ x_1}\,{ x_2}\, \left( k-1 \right) ^{3} \left( k+1 \right) ^{3}{k}^{2}q}{ \left( kq-1 \right)  \left( kq+1 \right)  \left( q{k}^{2}-1 \right)  \left( {k}^{4}q-1 \right) }}
    \\&+{\frac {{ x_3}\,{ x_1}\, \left( k-1 \right) ^{3} \left( k+1 \right) ^{3}{k}^{4}{q}^{2}}{ \left( q{k}^{2}-1 \right) ^{2} \left( q{k}^{2}+1 \right)  \left( {k}^{4}q-1 \right) }}+{\frac {{ x_1}\,{ x_2}\, \left( k-1 \right) ^{4} \left( k+1 \right) ^{4}q}{ \left( kq-1 \right)  \left( kq+1 \right)  \left( q{k}^{2}-1 \right) ^{2} \left( q{k}^{2}+1 \right)  \left( {k}^{4}q-1 \right) }}
    &\\E_{(0,0,2)}^{(2)} &= {{ x_3}}^{2}+{\frac {{{ x_2}}^{2} \left( k-1 \right)  \left( k+1 \right) }{ \left( kq-1 \right)  \left( kq+1 \right) }}+{\frac {{{ x_1}}^{2} \left( k-1 \right)  \left( k+1 \right) {k}^{2}}{ \left( q{k}^{2}-1 \right)  \left( q{k}^{2}+1 \right) }}+{\frac {{{ x_1}}^{2} \left( k-1 \right) ^{2} \left( k+1 \right) ^{2}}{ \left( kq-1 \right)  \left( kq+1 \right)  \left( q{k}^{2}-1 \right)  \left( q{k}^{2}+1 \right) }}
    &\\E_{(0,0,2)}^{(3)} &= \displaystyle {{ x_3}}^{2}-{\frac {{{ x_2}}^{2} \left( k-1 \right)  \left( k+1 \right)  G_{{1}}}{-{q}^{2}{ G_{{1}}}^{3}+k}}+{\frac {{{ x_1}}^{2} \left( k-1 \right)  \left( k+1 \right) { G_{{1}}}^{2}}{k{q}^{2}{ G_{{1}}}^{3}-1}}-{\frac {{{ x_1}}^{2} \left( k-1 \right) ^{2} \left( k+1 \right) ^{2}{ G_{{1}}}^{2}}{k \left( -{q}^{2}{ G_{{1}}}^{3}+k \right)  \left( k{q}^{2}{ G_{{1}}}^{3}-1 \right) }}
    &\\E_{(0,0,2)}^{(4)} &= {{ x_3}}^{2}-{\frac {{{ x_2}}^{2} \left( k-1 \right)  \left( k+1 \right) }{k{ G_{{2}}}-{q}^{2}}}+{\frac {{{ x_1}}^{2} \left( k-1 \right)  \left( k+1 \right)  G_{{2}}}{k{q}^{2}-{ G_{{2}}}}}-{\frac {{{ x_1}}^{2} \left( k-1 \right) ^{2} \left( k+1 \right) ^{2}}{k \left( k{ G_{{2}}}-{q}^{2} \right)  \left( k{q}^{2}-{ G_{{2}}} \right) }}
    &\\E_{(0,0,2)}^{(5)} &= {{ x_3}}^{2}-{\frac {{{ x_2}}^{2} \left( k-1 \right)  \left( k+1 \right) { G_{{2}}}^{2}}{k{ G_{{2}}}^{3}-{q}^{2}}}+{\frac {{{ x_1}}^{2} \left( k-1 \right)  \left( k+1 \right)  G_{{2}}}{k{q}^{2}-{ G_{{2}}}^{3}}}-{\frac {{{ x_1}}^{2} \left( k-1 \right) ^{2} \left( k+1 \right) ^{2}{ G_{{2}}}^{4}}{k \left( k{ G_{{2}}}^{3}-{q}^{2} \right)  \left( k{q}^{2}-{ G_{{2}}}^{3} \right) }}
\end{flalign*}
}

\section{Examples (symmetric)}\label{appendix: symmetric}

Finally, we give several examples of symmetric SSV polynomials $P_\mu^\enn$, computed by directly using Theorem~\ref{thm: symmetric} (with no terms combined). We make the same simplifications as in the previous appendix. 

{\small
    
\begin{flalign*}
    P_{(0,0,0)}^{(1)}  &={k}^{6}+2\,{k}^{4}+2\,{k}^{2}+1
    &\\P_{(0,0,0)}^{(2)}  &={k}^{6}+2\,{k}^{4}+2\,{k}^{2}+1
    &\\P_{(0,0,0)}^{(3)}  &={k}^{6}+2\,{k}^{4}+2\,{k}^{2}+1
    &\\P_{(0,0,0)}^{(4)}  &={k}^{6}+2\,{k}^{4}+2\,{k}^{2}+1
    &\\P_{(0,0,0)}^{(5)}  &={k}^{6}+2\,{k}^{4}+2\,{k}^{2}+1
\end{flalign*}
\begin{flalign*}
    P_{(1,0,0)}^{(1)}  &={x_1}\,{k}^{2}+{x_2}\,{k}^{2}+{x_3}\,{k}^{2}+{x_1}+{x_2}+{x_3}
    &\\P_{(1,0,0)}^{(2)}  &={x_3}\,{k}^{4}+{x_2}\,{k}^{3}G_1+{x_3}\,{k}^{2}+{x_1}\,{k}^{2}
    +{x_2}\,kG_1+{x_1}
    &\\P_{(1,0,0)}^{(3)}  &={x_3}\,{k}^{4}{G_1}^{2}+{x_2}\,{k}^{3}G_1+{x_3}\,{k}^{2}{G_1}^{2}+{x_1}\,{k}^{2}
    +{x_2}\,kG_1+{x_1}
    &\\P_{(1,0,0)}^{(4)}  &={x_3}\,{k}^{4}{G_1}^{2}+{x_2}\,{k}^{3}G_1+{x_3}\,{k}^{2}{G_1}^{2}+{x_1}\,{k}^{2}
    +{x_2}\,kG_1+{x_1}
    &\\P_{(1,0,0)}^{(5)}  &={x_3}\,{k}^{4}{G_1 }^{2}+{x_2}\,{k}^{3}G_1 +{x_3}\,{k}^{2}{G_1 }^{2}+{x_1}\,{k}^{2}
    +{x_2}\,kG_1 +{x_1}
\end{flalign*}
\begin{flalign*}
    P_{(1,1,0)}^{(1)} &= {x_1}\,{x_2}\,{k}^{2}+{x_1}\,{x_3}\,{k}^{2}+{x_2}\,{x_3}\,{k}^{2}+{x_1}\,{x_2}+{x_1}\,{x_3}+{x_2}\,{x_3}
    &\\P_{(1,1,0)}^{(2)} &= {x_2}\,{x_3}\,{k}^{4}+{x_1}\,{x_3}\,{k}^{3}G_1+{x_2}\,{x_3}\,{k}^{2}
    +{x_1}\,{x_2}\,{k}^{2}+{x_1}\,{x_3}\,kG_1+{x_1}\,{x_2}
    &\\P_{(1,1,0)}^{(3)} &= {x_2}\,{x_3}\,{k}^{4}{G_1}^{2}+{x_1}\,{x_3}\,{k}^{3}G_1+{x_2}\,{x_3}\,{k}^{2}{G_1}^{2}
    +{x_1}\,{x_2}\,{k}^{2}+{x_1}\,{x_3}\,kG_1+{x_1}\,{x_2}
    &\\P_{(1,1,0)}^{(4)} &= {x_2}\,{x_3}\,{k}^{4}{G_1}^{2}+{x_1}\,{x_3}\,{k}^{3}G_1+{x_2}\,{x_3}\,{k}^{2}{G_1}^{2}
    +{x_1}\,{x_2}\,{k}^{2}+{x_1}\,{x_3}\,kG_1+{x_1}\,{x_2}
    &\\P_{(1,1,0)}^{(5)} &={x_2}\,{x_3}\,{k}^{4}{G_1 }^{2}+{x_1}\,{x_3}\,{k}^{3}G_1 +{x_2}\,{x_3}\,{k}^{2}{G_1 }^{2}
    +{x_1}\,{x_2}\,{k}^{2}+{x_1}\,{x_3}\,kG_1 +{x_1}\,{x_2}
\end{flalign*}
\begin{flalign*}
    P_{(2,0,0)}^{(1)}  &={{x_1}}^{2}
    +{\frac {{x_1}\,{x_3}\, \left( k-1 \right)  \left( k+1 \right) q}{q{k}^{2}-1}}
    +{\frac {{x_1}\,{x_2}\, \left( k-1 \right)  \left( k+1 \right) {k}^{2}q}{q{k}^{4}-1}}
    +{\frac {{x_1}\,{x_2}\, \left( k-1 \right) ^{2} \left( k+1 \right) ^{2}q}{ \left( q{k}^{2}-1 \right)  \left( q{k}^{4}-1 \right) }}
    \\&+{{x_1}}^{2}{k}^{2}
    +{\frac {{x_1}\,{x_2}\, \left( k-1 \right)  \left( k+1 \right) {k}^{2}q}{q{k}^{2}-1}}
    +{\frac {{x_1}\,{x_3}\,{k}^{2} \left( k-1 \right)  \left( k+1 \right) q}{q{k}^{4}-1}}
    +{\frac {{x_1}\,{x_3}\,{k}^{4} \left( k-1 \right) ^{2} \left( k+1 \right) ^{2}{q}^{2}}{ \left( q{k}^{2}-1 \right)  \left( q{k}^{4}-1 \right) }}
    \\&+{{x_2}}^{2}
    +{\frac {{x_2}\,{x_3}\, \left( k-1 \right)  \left( k+1 \right) q}{q{k}^{2}-1}}
    +{\frac {{x_1}\,{x_2}\, \left( k-1 \right)  \left( k+1 \right) }{q{k}^{4}-1}}
    +{\frac {{x_1}\,{x_2}\,{k}^{2} \left( k-1 \right) ^{2} \left( k+1 \right) ^{2}q}{ \left( q{k}^{2}-1 \right)  \left( q{k}^{4}-1 \right) }}
    +{{x_2}}^{2}{k}^{2}
    \\&+{\frac {{x_1}\,{x_2}\,{k}^{2} \left( k-1 \right)  \left( k+1 \right) }{q{k}^{2}-1}}
    +{\frac {{x_2}\,{x_3}\,{k}^{4} \left( k-1 \right)  \left( k+1 \right) q}{q{k}^{4}-1}}
    +{\frac {{x_2}\,{x_3}\,{k}^{2} \left( k-1 \right) ^{2} \left( k+1 \right) ^{2}q}{ \left( q{k}^{2}-1 \right)  \left( q{k}^{4}-1 \right) }}
    +{{x_3}}^{2}
    \\&+{\frac {  {x_2}\,{x_3}\left( k+1 \right)  \left( k-1 \right)}{q{k}^{2}-1}}
    +{\frac {{x_1}\,{x_3}\,{k}^{2} \left( k-1 \right)  \left( k+1 \right) }{q{k}^{4}-1}}
    +{\frac {{x_1}\,{x_3}\, \left( k-1 \right) ^{2} \left( k+1 \right) ^{2}}{ \left( q{k}^{2}-1 \right)  \left( q{k}^{4}-1 \right) }}
    +{{x_3}}^{2}{k}^{2}
    \\&+{\frac {{x_1}\,{x_3}\,{k}^{2} \left( k-1 \right)  \left( k+1 \right) }{q{k}^{2}-1}}
    +{\frac {{x_2}\,{x_3}\,{k}^{2} \left( k-1 \right)  \left( k+1 \right) }{q{k}^{4}-1}}
    +{\frac {{x_2}\,{x_3}\,{k}^{4} \left( k-1 \right) ^{2} \left( k+1 \right) ^{2}q}{ \left( q{k}^{2}-1 \right)  \left( q{k}^{4}-1 \right) }}
    &\\P_{(2,0,0)}^{(2)}  &={{x_1}}^{2}{k}^{2}+{{x_2}}^{2}{k}^{2}+{{x_3}}^{2}{k}^{2}+{{x_1}}^{2}+{{x_2}}^{2}+{{x_3}}^{2}
    &\\P_{(2,0,0)}^{(3)}  &={{x_1}}^{2}+{{x_1}}^{2}{k}^{2}+{\frac {{{x_2}}^{2}k}{G_1}}+{\frac {{{x_2}}^{2}{k}^{3}}{G_1}}+{\frac {{{x_3}}^{2}{k}^{2}}{{G_1}^{2}}}
    +{\frac {{{x_3}}^{2}{k}^{4}}{{G_1}^{2}}}
    &\\P_{(2,0,0)}^{(4)}  &={{x_3}}^{2}{k}^{4}+{{x_2}}^{2}{k}^{3}G_{2}+{{x_3}}^{2}{k}^{2}
    +{{x_1}}^{2}{k}^{2}+{{x_2}}^{2}kG_{2}+{{x_1}}^{2}
    &\\P_{(2,0,0)}^{(5)}  &={{x_3}}^{2}{k}^{4}{G_2 }^{2}+{{x_2}}^{2}{k}^{3}G_2 +{{x_3}}^{2}{k}^{2}{G_2 }^{2}
    +{{x_1}}^{2}{k}^{2}+{{x_2}}^{2}kG_2 +{{x_1}}^{2}
\end{flalign*}
}

\end{document}